\documentclass[11pt,reqno]{amsart}

\usepackage[hidelinks]{hyperref}
\usepackage[foot]{amsaddr}
\usepackage{enumerate}
\usepackage{scalerel,stackengine}
\usepackage{csquotes}
\usepackage{amsmath}
\usepackage{amssymb}
\usepackage{mathrsfs} 
\usepackage{mathtools}
\usepackage[markup=nocolor]{changes}
\usepackage{a4wide}
\usepackage{caption}
\usepackage{subcaption}
\usepackage{cancel}
\usepackage{cleveref}
\usepackage{graphicx}

\numberwithin{equation}{section}

\newcommand{\lap}{\Delta}
\newcommand{\mollif}{\eta}
\newcommand{\delOm}{\partial\Omega}
\newcommand{\flap}{\ensuremath{(-\Delta)^s}}
\DeclareMathOperator{\halfflap}{\ensuremath{(-\Delta)^{s/2}}}
\newcommand{\R}{\ensuremath{\mathbb{R}}}

\newcommand{\Rd}{\ensuremath{{\R^d}}}

\newcommand{\Zd}{\ensuremath{{\Z^d}}}
\newcommand{\I}{\ensuremath{\mathcal{I}}}

\newcommand{\csob}{C_{\text{Sob.}}}
\newcommand{\cbound}{C_{\text{Bndry.}}}
\newcommand{\capprox}{C_{\text{Approx.}}}
\newcommand{\ci}{C_{\text{(i)}}}
\newcommand{\cii}{C_{\text{(ii)}}}
\newcommand{\ciii}{C_{\text{(iii)}}}
\newcommand{\N}{\ensuremath{\mathbb{N}}}
\newcommand{\Z}{\ensuremath{\mathbb{Z}}}
\newcommand{\V}{\ensuremath{\mathbb{V}}}

\newcommand{\eps}{\varepsilon}
\newcommand{\dx}{\mathop{\mathrm{d}x}}
\newcommand{\dt}{\mathop{\mathrm{d}t}}
\newcommand{\dr}{\mathop{\mathrm{d}r}}
\newcommand{\dy}{\mathop{\mathrm{d}y}}

\newcommand{\doxi}{\mathop{\mathrm{d}o}(\xi)}
\newcommand{\dom}{\mathop{\mathrm{d}\omega}}

\renewcommand{\i}{\mathrm{i}}
\newcommand{\e}{\mathrm{e}}

\newcommand{\abs}[1]{\ensuremath{\left|{#1}\right|}}
\newcommand{\norm}[1]{\ensuremath{\left\|{#1}\right\|}}
\newcommand{\Hs}[1]{\ensuremath{\left|{#1}\right|_{H^s(\R^d)}}}
\newcommand{\Hsnorm}[1]{\ensuremath{\left\|{#1}\right\|_{H^s(\R^d)}}}
\newcommand{\tildeHsOm}{\ensuremath{\widetilde H^{s}(\Omega)}}
\newcommand{\tildeHtOm}{\ensuremath{\widetilde H^{t}(\Omega)}}

\newcommand{\Hms}[1]{\ensuremath{\norm{#1}_{H^{-s}(\R^d)}}}
\newcommand{\Ht}[1]{\ensuremath{|{#1}|_{H^t(\R^d)}}}
\newcommand{\Lt}[1]{\ensuremath{\left\|{#1}\right\|_{L^2(\R^d)}}}
\newcommand{\resOm}[1]{\ensuremath{{#1}\big|_{\Omega_h}}}

\newcommand{\Poincare}{Poincar\'e}

\DeclareMathOperator{\supp}{supp}
\DeclareMathOperator{\dist}{dist}
\DeclareMathOperator{\sinc}{sinc}
\DeclareMathOperator{\FT}{\mathcal F}
\DeclareMathOperator{\IFT}{\mathcal F^{-1}}
\DeclareMathOperator{\PiN}{\Pi_N}
\DeclareMathOperator{\Ns}{\mathcal N_s}

\definecolor{dk_blue}{rgb}{.0,.0,.7}

\newtheorem{theorem}{Theorem}[section]
\newtheorem{lemma}{Lemma}[section]

\newtheorem{proposition}{Proposition}[section]
\newtheorem{assumption}{Assumption}[section]
\newtheorem{remark}{Remark}[section]

\hypersetup{
    colorlinks=false, 
		linktocpage=true,
    linktoc=all
}

\title{Analysis of a $\sinc$-Galerkin Method for the Fractional Laplacian}
\thanks{
HA and LS are partially supported by NSF grants DMS-2110263, DMS-1913004, the Air
Force Office of Scientific Research under Award NO: FA9550-22-1-0248.
LS gratefully acknowledges a doctoral scholarship from the Friedrich-Ebert-Stiftung.
}

\author{Harbir Antil$^\ast$}
\address[$^\ast$]{Department of Mathematical Sciences and the Center for Mathematics and Artificial Intelligence (CMAI), George Mason University, Fairfax, VA 22030, USA.}
\email{hantil@gmu.edu}

\author{Patrick Dondl$^\dagger$}

\author{Ludwig Striet$^\dagger$}
\address[$\dagger$]{Abteilung für Angewandte Mathematik, Albert-Ludwigs-Universität Freiburg, Hermann-Herder-Straße 10, 79104 Freiburg i. Br. }
\email{patrick.dondl@mathematik.uni-freiburg.de}
\email{ludwig.striet@mathematik.uni-freiburg.de}

\begin{document}
\maketitle

\begin{abstract}
	We provide the convergence analysis for a $\sinc$-Galerkin method to solve
	the fractional Dirichlet problem. This can be understood as a follow-up of \cite{ADS2021}, 
	where the authors presented a $\sinc$-function based method to solve fractional
	PDEs. While the original method was formulated as a collocation method, we
	show that the same method can be interpreted as a nonconforming Galerkin method, 
	giving access to abstract error estimates. Optimal order of convergence 
	is shown without any unrealistic regularity assumptions on the solution.
\end{abstract}

\section{Introduction}
Motivated by a number of applications, we aim to solve the fractional
Dirichlet problem, i.e., find $u$ in an appropriate space such that
\begin{equation}
	\label{eq:bvp}
	\flap u = f\text{ in }\Omega,\quad u \equiv 0\text{ in }\R^d\setminus\Omega.
\end{equation}
We are dealing with the integral definition of the fractional Laplacian which -- in
terms of the Fourier transformation -- can be defined as
\begin{equation}
    \label{eq:fourier_def_flap}
	\FT\left(\flap u\right)(\omega) = |\omega|^{2s} \hat u(\omega),
\end{equation}
where $\hat u (\omega) = \left(\FT u\right)(\omega)$ is the Fourier transformation of $u$.
If defined on the whole $\Rd$, at least ten different, but equivalent definitions of 
the fractional Laplacian exist \cite{Kwasnicki2015}. One of the definitions
that we want to mention here is the \emph{integral definition} which reads \cite{Dinezza2012}
\begin{equation}
	\label{eq:integral_def_flap}
	\flap u = C(d,s) \text{P.V.} \int_\Rd \frac{u(x) - u(y)}{\abs{x-y}^{d+2s}}\dy 
\end{equation}
and exposes the nonlocal character of the fractional Laplacian: 
to evaluate
$\flap u(x)$ at one point $x\in\Rd$, a singular integral over the whole $\Rd$ 
has to be evaluated. This makes numerical approximations of \cref{eq:bvp}
challenging. A proof for the equivalence of \cref{eq:fourier_def_flap} and \cref{eq:integral_def_flap},
as well as the exact value of the constant $C(d,s)$, can be found -- among others --
in \cite{Dinezza2012}. 

Various approaches to solve such problems have been proposed in the last years.
The article \cite{Borthagaray2017}, provides the analysis and
implementation of a finite element Galerkin method for the fractional ($s$)-Poisson problem on
Lipschitz domains. For $\mathbb P_1$-finite elements, this leads to a conforming
finite element method assuming $s\in(0,1)$. 
Further studies on finite element methods for the fractional Dirichlet
problem can be found for example in \cite{Ainsworth2017,Ainsworth2018,Bonito2018,Bonito2019}.
 Some further examples of approximations of nonlocal problems can be found in \cite{Delia2013,Karkulik2019}. 

Finite difference methods have been proposed 
in, e.g.,
\cite{Huang2016}. Such finite-difference formulations of fractional PDE
on Cartesian grids lead to linear systems with matrices that can be efficiently
applied to vectors using fast Fourier transformation (FFT) based algorithms. Examples
of algorithms of this type can be found in \cite{Minden2020} or in \cite{Duo2018}
and, of course, our aforementioned work. However, the analysis for such
algorithms typically needs strong regularity
assumptions or is limited to rectangular domains.

In \cite{ADS2021}, we proposed a collocation method to solve \cref{eq:bvp}
where we used a $\sinc$-basis to solve the equation pointwise on a grid.
More concretely, we find a function 
\begin{equation}
	u_h(x) = \sum\limits_{k\in\Zd} u_k\varphi^N_k(x),\quad u_k\in\R
\end{equation}
satisfying 
\begin{equation}
	\label{eq:formulation_collocation}
\begin{aligned}
	\flap u_h(x_k) &= f_k & \text{ if } x_k = hk\in\Omega \\
		u_h(x_k) &= 0 & \text{ if }x_k = hk\not\in\Omega\,.
\end{aligned}
\end{equation}
The functions $\varphi^N_k(x)$ are $\sinc$-functions which we scale by a
parameter $N$ and shift by $k$ gridpoints. The natural number $N > 0$ is the 
number of grid points we use in each spatial direction and $h\coloneqq 1/N$ is the
grid spacing. We will interchangeably use $N$ or $h = 1/N$ depending on the context. 
The equation $u_h(x_k) = 0$ if $x_k\not\in\Omega$ is equivalent to
$u_k = 0$ if $x_k\not\in\Omega$ as the basis functions fulfill the Kronecker-delta
property $\varphi^N_k(x_j) = \delta_{k,j}$. We detail this in
\cref{sub:description_sinc_collocation}, as well as we briefly recall
the $\sinc$-collocation method and its efficient implementation we presented in
the aforementioned paper.

Our approach to solve the fractional Dirichlet problem, formulated as a
collocation method, can be viewed also as a finite difference method.
 For the case $d = 1$,
this has been established in \cite{Huang2016}. As we pointed out in \cite{ADS2021},
our approach for solving \cref{eq:bvp} in more than one dimension can be seen
as a multi-dimensional generalization of this. Other works in this direction can be
found e.g. in the work \cite{Huang2014} by the same authors as \cite{Huang2016} where
a finite difference scheme based on quadrature of the singular integrals is defined;
in \cite{Han2022} where a similar technique as in \cite{Huang2014} is used to 
obtain a discretization of the operator on unstructured grids. 

Another approach that is related to ours is \cite{Hao2021}. They define a discrete
convolution operator whose coefficients are obtained via cubature in Fourier
space. This approach can be extended to non-rectangular domains using the fictious
domain method.

 As we detail in the following,
our method can also be seen as a Galerkin method using a $\sinc$-basis consisting of  
$C^\infty$-functions with unbounded support.  In stark contrast to, e.g., \cite{Borthagaray2017},
this leads to a nonconforming Galerkin method. Nevertheless, similarly to
\cite{Borthagaray2017}, no additional regularity
of the solution is assumed other than what has been shown in the literature for
such problems.

In our aforementioned work, we could demonstrate the effectiveness
of our method on a benchmark problem where $\Omega$ is a ball in $\Rd$ and
$f\equiv 1$ in $\Omega$. We also used it to solve problems with 
high spatial resolution, with a minimal computational expense. Indeed, 
our method has  quasi-linear   complexity, which is also visible in our computational
results.
For example, solving \cref{eq:bvp} with $s = 1/2$ and about $6\cdot
10^6$ grid points in $\R^3$ takes approximately $12$ minutes wall clock time on a standard office work
station. This includes the setup of the operator (approx.\ $10$ minutes) and the solve
of the system (approx.\ $2$ minutes), which we improved in comparison to the implementation
presented in \cite{ADS2021}, with the use of a simple periodic fractional Laplacian
as a preconditioner.

However, a numerical analysis of the method presented in \cite{ADS2021} was
still pending. This paper aims to close this gap. We want to
emphasize at this point that all the implementation details as well as the
pseudo-code that we already provided in \cite{ADS2021} may remain unchanged.  We show
quasi-optimal rates of convergence in terms of the energy norm which are in agreement with the
results that we obtained in numerical experiments.
 Briefly, our main result is that for $f\in C^{\frac{1}{2}}(\overline{\Omega})$
and under suitable assumptions on the domain $\Omega$,
the solution $u_h$ we obtain via \cref{eq:formulation_collocation}, taking $f_k = f(x_k)$ for $x_k\in \Omega$, fulfills the estimate 
\begin{equation}
	\label{eq:general_error_bound}
	\norm{u-u_h}_{H^{s}(\Rd)} \leq C(d, s, \Omega) \,\norm{f}_{C^{\frac{1}{2}}(\overline{\Omega})} \,\abs{\log{h}} h^{1/2}
\end{equation}
where $u$ is the exact solution to the problem and $h$ is the grid spacing.
We show this in \cref{thm:main_ii}.
 In our analysis, we need to assume that there exists a Hölder continuous extension
of $f$ to a domain that is slightly larger than $\Omega$. Thus,
we require $f\in C^\beta(\overline\Omega)$ for an appropriate $\beta$.
When $\beta\ge \frac{1}{2}$, the above result applies directly as then the result of an extension-and-mollification procedure applied to $f$ agrees sufficiently well with the original function. In cases where $f$ has more limited
regularity, we need to replace the discrete right-hand side $f_k$ by an evaluation of a suitable extended-and-mollified right hand side. For this approximation, we show that \cref{eq:general_error_bound}
holds in \cref{thm:main_i}.
 In our proofs, we make substantial use of the regularity
of the solutions shown in \cite{Borthagaray2017,Borthagaray2018,RosOton2014}. The only additional assumptions
we have to make concern the domain and the right-hand side.
We assume that one can slightly enlarge the domain and, when considering \cref{eq:bvp}
on the
enlarged domain, have similar regularity of solutions as for the problem
on the
original domain. Details are given in \cref{sub:domain_approximation}.

To establish the rates of convergence, we 
show how the collocation method can be equivalently written as a Galerkin method. As the resulting method is nonconforming, we must apply a Strang type lemma, see e.g. \cite{Brenner2007}, and bound
the resulting error terms appropriately. This reduces to showing that the chosen $\sinc$-basis
can approximate the elements of the fractional Sobolev space $\tildeHsOm$ as well as the elements
of its dual.

The remainder of this paper is structured as follows.
In \cref{sec:summary_related_results},
we summarize results of related work that this article is based on. This
includes a brief overview of the $\sinc$-collocation method we presented in
\cite{ADS2021} as well as the weak formulation of \cref{eq:bvp} and results
from regularity theory that we need.

In \cref{sec:formulation_sinc_galerkin_method}, we provide the formulation
of the $\sinc$-Galerkin method. This includes the definitions of the grid-dependent discrete
spaces as well as some properties of them, such as for example a \Poincare \
inequality, and inverse-estimate type result and properties of the norms on the
spaces. Further, we present the abstract error estimate that we use here.

In \cref{sec:approximation_errors}, we compute the approximation errors in
the discrete spaces which, using the abstract error estimate, immediately
yields the optimal rate of convergence for our numerical method in fractional
Sobolev norms. For this, we must make some additional assumptions on the regularity
of the domain.

Finally, \cref{sec:main_theorems}, contains our main theorems.

\section{Notation and Preliminary Results}
\label{sec:summary_related_results}
We briefly recapitulate known results that will be useful for the present work. In particular, we discuss the
$\sinc$-basis collocation method, as well as regularity theory.

\subsection{The \texorpdfstring{$\sinc$}{sinc}-collocation method}
\label{sub:description_sinc_collocation}
In \cite{ADS2021}, we proposed to approximate \cref{eq:bvp} in a $\sinc$-basis.
Using the $d$-fold Cartesian product of $\sinc$-functions, we defined a basis
function
\begin{equation}
	\varphi(x) \coloneqq \prod \limits_{i=1}^d \sinc(x_i),\quad \sinc(x)\coloneqq \frac{\sin(\pi x)}{\pi x},\quad
		x = (x_1\cdots x_d)^T\in\Rd,
\end{equation}
and its dilated and shifted versions,
\begin{equation}
	\varphi^N_k(x) = \varphi(Nx-k),\quad k = (k_1 \cdots k_d)^T\in\Z^d,
\end{equation}
which span our discrete space.
Throughout this work, we use the conventions 
\begin{equation}
	\FT u(\omega) = (2\pi)^{-d} \int_\Rd u(x) \e^{-i\omega x}\dx, \qquad
		\IFT \hat u(x) = \int_\Rd \hat u(\omega) \e^{\i\omega x}\mbox{d}\omega.
\end{equation}
For this specific scaling of the Fourier transformation, we have the Plancherel identity
\begin{equation}
	(v, w)_{L^2(\Rd)} = (2\pi)^d (\hat v, \hat w)_{L^2(\Rd)}.
\end{equation}
A well-known result  (see, e.g., \cite{Stenger2011}) is that
\begin{equation}
	\sinc(x) = \frac{1}{2\pi}\int_{[-\pi,\pi]} \e^{\i\omega x}\dom = \frac{1}{2\pi} \IFT \chi_{[-\pi,\pi]}(\omega),
\end{equation}
so that $\FT\varphi(\omega) = \frac{1}{(2\pi)^d} \chi_{[-\pi,\pi]^d}(\omega)$ and, using
the Fourier shift and scaling theorems,
\begin{equation}
	\widehat{\varphi^N_k}(\omega) = N^{-d}  \e^{-\i \omega k/N }\hat\varphi (\omega/N)
		= (2N\pi)^{-d} \e^{-\i\omega  k/N} \chi_{D_N}(\omega),\quad D_N \coloneqq [-N\pi,N\pi]^d.
\end{equation}
In order to solve \cref{eq:bvp}, we proposed in \cite{ADS2021} to solve the
problem pointwise on the grid $h\Z^d$, $h \coloneqq 1/N$, in the sense of a collocation
method. As presented in \cref{eq:formulation_collocation}
this formulation leads to the linear system 
\begin{equation}
	\label{eq:discrete_linear_system}
	S_\Omega\Phi^NS_\Omega^T \vec u = \vec f,
\end{equation}
and we provided efficient strategies to assemble the system matrix $\Phi^N$
and to solve the resulting system. In \cref{eq:discrete_linear_system},
$\vec u = (u_k)$ is the sought-after coefficient vector of the solution $u_h$ and
$\vec f = (f_k)_{k/N\in\Omega}$ is the discrete right-hand side which we obtained
as pointwise evaluations of $f$ within $\Omega$ for the numerical experiments in \cite{ADS2021}.
$S_\Omega$ is a linear operator which
leaves the coefficients within the domain unchanged and sets the others to
zero, thus restricting the domain of computation. When we numerically
solve this system, this approach is the projected conjugate gradient method,
alternatively our method can be viewed as a constrained quadratic
optimization problem which we solve using a null-space approach.

The application of the system matrix to a coefficient vector $\vec u$ is essentially the discrete 
convolution with the fractional Laplacian of the basis function $\varphi$ above, namely
\begin{equation}
	\label{eq:discrete_operator}
    (\Phi^N \vec u)_j = \sum\limits_{k\in \{0\ldots N-1\}^d} u_k \Phi^N(j-k),\quad
        \Phi^N(j-k) = N^{2s} \flap \varphi(j-k).
\end{equation}
This convolution can be computed efficiently using the discrete Fourier transformation, leading to a numerical complexity of $\mathcal O(N^{d}\cdot\log(N^{d}))$.

\subsection{Weak formulation of the boundary value problem}
\label{sub:weak_formulation_bvp}
The weak formulation of \cref{eq:bvp} reads as follows \cite{RosOton2014, Borthagaray2017, Bonito2019}.
 Find $u\in \V$ such that 
\begin{equation}
	\label{eq:weak}
	a(u,v) = F(v) \quad \text{for all $v\in \V$}
\end{equation}
where $\V$ is the space
\begin{equation}
	\label{eq:def_V}
	\V \coloneqq \tildeHsOm = \{v\in H^s(\Rd) \, | \, v = 0 \ {\rm in } \ \mathbb{R}^d \setminus \Omega\}
\end{equation}
and $H^s(\Rd)$ is a fractional Sobolev space \cite{Dinezza2012}.
The bilinear form $a : \V\times\V\longrightarrow \R$  reads
\begin{equation}
	a(u, v) = \int\limits_\Rd (-\lap)^{\frac{s}{2}} u\cdot (-\lap)^{\frac{s}{2}} v\dx,
\end{equation}
and the linear form $F(\cdot)$ is given by
\begin{equation}
	F(v) = \int_{\Omega} f v\dx
\end{equation}
for $f \in L^2(\Omega)$.
Using the Plancherel identity, the bilinear form can be evaluated in the
Fourier domain, which yields
\begin{equation}
	\label{eq:def_a_on_V}
	a(u, v) = 
	(2\pi)^d\int\limits_{\R^d} |\omega|^{2s} \hat u {\hat v}^*\dom
\end{equation}
where  the asterisk  denotes the complex conjugate. The fractional Sobolev seminorm
for $u\in H^s(\Rd)$ is defined as
\begin{equation}
	\label{eq:def_frac_seminorm}
	\Hs{u} \coloneqq a(u, u)^{1/2} =
		\left( (2\pi)^{d/2}\int_\Rd \abs{\omega}^{2s}\abs{\hat u}^2 \right)^{1/2}.
\end{equation}
Given a suitable finite dimensional subspace $V_h\subset \V$, an abstract Galerkin
method can be formulated as follows. Find $u_h\in V_h$ such that
\begin{equation}
	\label{eq:abstract_galerkin_formulation}
	a(u_h, v_h) = F(v_h)\quad\text{for all $v_h\in V_h$}.
\end{equation}
This Galerkin formulation can be used as to derive a finite element method (see, e.g., \cite{Borthagaray2017}); a non-conforming version thereof also forms the basis of our convergence analysis.
\subsection{Summary of results from the literature}
\label{sub:summary_regularity_results}
For our analysis, we require certain regularity and decay properties of solutions to \cref{eq:bvp}
and some mathematical key tools.  In particular, we require
 \cite[Lemma 2.7]{RosOton2014}, which we repeat here for the reader's convenience.
\begin{theorem}[Lemma 2.7 in \cite{RosOton2014}]
	\label{thm:bound_u_boundary}
Let $\Omega$ be a bounded domain satisfying the exterior ball condition and let $f\in L^\infty(\Omega)$. Let $u$ be the solution of \cref{eq:bvp}. Then, 
\begin{equation}
    \label{eq:bound_u_boundary}
    \abs{u(x)}\leq C \norm{f}_{L^\infty(\Omega)} \dist(x,\Omega^c)^s
\end{equation}
where $C = C(\Omega, s)$ is a constant depending only on $\Omega$ and $s$.
\end{theorem}
Acosta and Borthagaray \cite{Borthagaray2018, Borthagaray2017} use global Hölder regularity of solutions (also proved in \cite{RosOton2014}) to derive estimates for $u$ in fractional Sobolev norms. Again, we repeat the result here for the reader's convenience.
\begin{theorem}[Proposition 2.4 in \cite{Borthagaray2018}]
	\label{thm:sobolev_reg_u_lipschitz}
	Let $\Omega$ be a Lipschitz domain satisfying the exterior ball condition and consider
	$\beta = 1/2 - s$ if $s < 1/2$ or $\beta > 0$ if $s\geq 1/2$. Then, if $f\in C^\beta(\Omega)$,
	the solution $u$ of \cref{eq:bvp} belongs to
	$\widetilde{H}^{s+1/2-\theta}(\Omega)$ with
	\begin{equation}
	    \label{eq:bound_u_sobolev}
		\norm{u}_{\widetilde{H}^{s+1/2-\theta}(\Omega)} \leq \frac{C(\Omega, s, d)}{\theta}
			\norm{f}_{C^{\beta}(\Omega)}
	\end{equation}
 for any $\theta > 0$.
\end{theorem}
If $\Omega$ is smooth, no Hölder-regularity of $f$ is needed in order to have Sobolev-regularity
of the solution $u$. In fact, in that case, we have the following theorem.
\begin{theorem}[Proposition 2.7 in \cite{Borthagaray2018}]
	\label{thm:sobolev_reg_u_smooth}
	Let $\Omega$ be a smooth domain, $f\in H^r(\Omega)$ for $r\geq -s$ and $u\in\tildeHsOm$ 
	be the solution of the Dirichlet problem \cref{eq:bvp}. Then, the following regularity
	estimate holds \[
		\abs{u}_{H^{s+\alpha}(\Rd)} \leq C(d,\alpha)\norm{f}_{H^{r}(\Omega)}.
	\]
	Here, $\alpha = s+r$ if $s+r < \frac{1}{2}$ and $\alpha = \frac{1}{2} - \theta$ 
	if $s+r\geq \frac{1}{2}$, with $\theta > 0$ arbitrarly small.
\end{theorem}
\begin{remark}
\label{rem:def_cbound_csob}
In the following, approximation errors depend on the decay of the solution
to the boundary in the sense of \Cref{thm:bound_u_boundary} and on the $\abs{\,\cdot\,}_{H^{t}(\Rd)}$-seminorm of the solution $u$ for some $t > s$ in the sense of \Cref{thm:sobolev_reg_u_lipschitz,thm:sobolev_reg_u_smooth}.
To simplify the statements, we write
\[
	u(x) \leq \cbound(\Omega, f, s) \dist(x,\Omega^c)^s
\]
and
\[
	\norm{u}_{H^{t}(\Rd)} \leq \csob(\Omega, f, s)
\]
for the solution $u$ to the exterior value problem on a suitable domain $\Omega$
and a given right-hand side $f$. The constant $\csob(\Omega, f, s)$ has to be chosen
appropriately, depending on the given domain and right-hand side, in the sense of
	\Cref{thm:sobolev_reg_u_lipschitz,thm:sobolev_reg_u_smooth}.  We note that other 
	results regarding  Sobolev regularity, as, e.g., in \cite{Abels2023}, 
		may be used to obtain our convegence rates as well. This could lead to different requirements on the domain regularity.
\end{remark}
Another key ingredient for our proofs is the \emph{Poisson summation formula (PSF)}
that can be stated as follows 
\begin{theorem}[Poisson summation formula] Suppose that $u,\hat u \in L^1(\Rd)$
	\label{thm:poisson_summation}
	satisfy 
	\begin{equation}
		\abs{u(x)} + \abs{\hat u(x)} \leq C \left(1+\abs{x}\right)^{-d-\delta}
	\end{equation}
	for some $C,\delta > 0$. Then $u$ and $\hat u$ are both continuous and for all $x\in\Rd$ we 
	have
	\begin{equation}
\frac{1}{(2N\pi)^d}\sum\limits_{k\in\Zd} u(k/N) \e^{-\i k/N}
	= \sum\limits_{k\in\Zd}\hat u(2N\pi k + \omega)
	\end{equation}
\end{theorem}
For a proof, see Theorem 3.1.17 in \cite{Grafakos2010} after appropriately
rescaling to match the conventions of the Fourier transformations used there
and exchanging the roles of $\FT$ and $\IFT$ or \cite{Nguyen2017}.

\section{Formulation of the \texorpdfstring{$\sinc$}{sinc}-Galerkin method}
We now formulate a Galerkin method on the space of $\sinc$-functions, which yields
the same solutions as the $\sinc$-collocation method introduced earlier.
\label{sec:formulation_sinc_galerkin_method}
\subsection{Discrete Function Spaces}
As a first discrete space, we define
\begin{equation}
	\V_h(\Rd) = \left\{ v_h\Big|
		v_h = \sum\limits_{k\in\Zd} v_k\varphi^N_k(x),\, k\in\Zd
		\right\} \cap H^s(\Rd)
\end{equation}
using scaled and shifted $\sinc$-functions $\varphi^N_k(x)$ as our basis. 
To actually implement the method, we need a finite-dimensional subspace of $\V_h(\Rd)$. We take
\begin{equation}
	\V_h(\Omega) = \left\{
		v_h = \sum\limits_{k\in\Omega_h} v_k\varphi^N_k(x)
		\right\},\quad \Omega_h = \left\{k\in\Zd\text{ s.t. }k/N\in\Omega\right\}.
\end{equation}
Denoting grid points by $x_k \coloneqq k/N$, we mention that for $k,l\in\Zd$ we have $\varphi^N_k(x_k) = 1$ and $\varphi^N_k(x_l) = 0$ if $k\neq l$. 
We implicitly extend functions $v\in\V_h(\Omega)$ to functions in $\V_h(\Rd)$ by setting the coefficients in $h\Zd\setminus\Omega_h$ to $0$.

Complementary to the extension, we define a \emph{discrete restriction operator} $\resOm{\cdot} : \V_h(\Rd)\rightarrow\V_h(\Omega)$
via
\begin{equation}
	\label{eq:def_discrete_restriction}
	\resOm{v_h}(x) = \sum\limits_{k\in\Omega_h} v_k\varphi^N_k(x).
\end{equation}
 We emphasize that the restriction operator $\resOm{\cdot}$ does not actually restrict the support
 of the continuous function $v_h$ to $\Omega$, but only the sum of the coefficients
 to the coefficients with index $k\in\Omega_h$.

\subsubsection{Orthogonality properties of the discrete spaces}
We first note that the basis functions $\varphi^N_k, k\in\Zd$ form an $L^2$-orthogonal system.
Indeed, we have, for $k, j\in\Zd$,
\begin{align}
	\left(\varphi^N_k, \varphi^N_j\right)_{L^2(\Rd)} &=
			(2\pi)^d\int_{\Rd} \hat\varphi^N_k(\omega) {\hat\varphi^N_j(\omega)}^*\dom 
			= (2\pi)^{d} (2N\pi)^{-2d}  \int_{D_N}\e^{-i\omega k/N}\cdot \e^{i\omega j/N}\dom \nonumber\\ 
			&= N^{-d} \int_{[-1/2,1/2]^d} \e^{-i 2\pi \omega (k-j)}\dom = N^{-d} \delta_{j,k},
\end{align}
where $\delta_{j,k}$ is a Kronecker delta on $\Zd$  and, as before, the asterisk
denotes the complex conjugate. 
As a useful consequence, the $L^2$-inner product of elements $v_h, w_h\in\V_h(\Rd)$
is the $\ell^2$-inner product of their coefficient vectors appropriately scaled.
Namely, for $v_h, w_h\in\V_h(\Rd)$,
\begin{equation}
	(v_h, w_h)_{L^2(\Rd)} = N^{-d} \sum\limits_{k\in\Zd} v_k w_k.
\end{equation}
The $L^2$-norm of $v_h$ can thus be written as
\begin{equation}
	\label{eq:sinc_L2_l2_identity}
	\Lt{v_h} = \Big(N^{-d} \sum\limits_{k\in\Zd} v_k^2\Big)^{1/2}.
\end{equation}

For a function $v_h\in\V_h(\Rd)$, we denote by $\vec v$ its coefficient
vector $(v_k)_{k\in\Zd}$,
which we can use to express \cref{eq:sinc_L2_l2_identity} as
\begin{equation}
	\Lt{v_h} = N^{-d/2}\norm{\vec v}_{\ell^2}
\end{equation}
where $\norm{\cdot}$ is the standard $\ell^2$-norm.

\subsubsection{Interpolation Operator}
We define an interpolation operator $\PiN$ via 
\begin{equation}
	\label{eq:def_interp_operator}
\PiN : C^{\infty}(\Rd) \cap L^2(\Rd) \longrightarrow \V_h(\Rd),\quad
	\big(\PiN v\big)(x) = \sum\limits_{k\in\Zd}v(k/N)\varphi^N_k(x).
\end{equation}
If $\supp\hat v\subseteq D_N$, then we have the pointwise equality
\begin{equation}
	\label{eq:interp_equal_bandlimited}
	\big(\PiN v\big)(x) = v(x),
\end{equation}
and if $\hat v$ decays exponentially, then the $\sinc$-interpolation
converges exponentially to the function in the $L^\infty$-norm, see \cite{Stenger2011}.
We provide error estimates in fractional Sobolev norms in \cref{sub:interpolation_errors_u,sub:approx_rhs}.

\begin{remark}
It is worth noting here that the discrete operator $\Phi^N$ 
	is the \emph{exact} fractional Laplacian on the space $\V_h(\Rd)$ in the sense
that for $v_h\in\V_h(\Rd)$
\begin{equation}
	\label{eq:flap_on_discr_space}
	\overrightarrow{\flap v_h} = \Phi^N \vec{v}.
\end{equation}
One can see this by first noting that if $\supp \hat{v}_h \subseteq D_N$, then
also $\supp \widehat{\flap v_h} \subseteq D_N$. Therefore, $\flap v_h$ can be reconstructed
exactly using the $\sinc$-interpolation operator in the sense of
\cref{eq:def_interp_operator,eq:interp_equal_bandlimited}, and the coefficients
are given by
\begin{equation}
	\flap v_h(x_j) = \left(\Phi^N \vec v\right)_j
\end{equation}
where $\Phi^N$ is the operator defined in \cref{eq:discrete_operator}. The useful identity
\begin{equation}
	N^{-d} (\Phi^N\vec v, \vec w)_{\ell^2} = (\flap v_h, w_h)_{L^2(\Rd)}\quad\forall w_h\in\V_h(\Rd)
\end{equation}
follows immediately.  As a special case, we obtain that
\begin{equation}
    N^{-d} (\vec f, \vec w)_{\ell^2} = (f_h, w_h)_{L^2(R^d)} = (\flap u_h, w_h)_{L^2(R^d)}
			= N^{-d} (\Phi^N \vec u, \vec w)_{\ell^2}\quad\forall w_h\in\V_h(\Rd)
\end{equation}
where $\vec u$ is the solution to the discrete system of equations \cref{eq:discrete_linear_system},
$u_h$ is the $\sinc$-function with coefficient vector $\vec u$ and
$f_h$ is a $\sinc$-function whose values at the grid points within $\Omega$ are precisely the entries 
of the right-hand vector $\vec f$ in the linear system.
\end{remark}

Different from \cite{ADS2021}, we consider the operator $\Phi^N$ here as an
operator on the infinite dimensional space $\V_h(\Rd)$ to simplify our
notation. The idea is that for any element $v_h \in\V_h(\Rd)$, the definition
\begin{equation}
	(\Phi^N\vec v)_j = \sum\limits_{k\in\Zd} v_k \Phi^N(j-k)
\end{equation}
is reasonable. For an implementation it is of course necessary that only finitely many of the $v_k$ are non-zero, for example due to a  truncation in the sense of \cref{eq:discrete_linear_system}.

To be able to reuse our implementations from \cite{ADS2021}, we assume in the following
$\Omega \subset [0,1)^d$, but point out that the method can be implemented on other cubes or rectangles
as well.

\subsubsection{Poincar\'e inequality and inverse estimates}
On the finite dimensional discrete space $\V_h(\Omega)$, we can establish a Poincar\'e inequality, thus ensuring that the bilinear form $a(\cdot, \cdot)$ is coercive on 
$\V_h(\Omega)$. 
\begin{lemma}[Poincar\'e inequality]
	\label{lemma:poincare}
	Let $\Omega\subset\Rd$ a bounded open set.
	Then, there exists a constant $0 < C\in\R$ such that for all $v_h\in\V_h(\Omega)$,
	\begin{equation}
		\label{eq:poincare}
		\Lt{v_h} \leq C \Hs{v_h}
	\end{equation}
	
	where $C = C(\Omega, d, s)$ depends
	on $\Omega, d$ and $s$,
	but not on $N$.
	
\end{lemma}
\begin{proof}
	 Let $\Omega_h \coloneqq \{k\in\Zd : k/N\in\Omega\}$. 
	 As $\Omega$ is bounded, there exists a constant $C_\Omega > 0$ such that
	\[
		\left|\{k\in\Zd : k/N\in\Omega\}\right| = \#\text{Elements in }\Omega_h \leq C_\Omega N^d.
	\]
	Let $v_h = \sum\limits_{k\in\Omega_h} v_k\varphi^N_k\in\V_h(\Omega)$.
	For $\omega\in D_N$, we compute
	 
	\begin{align}
	\abs{\hat v_h(\omega)}^2 &= \Big|
	 (2N\pi)^{-d}\sum\limits_{k\in \Omega_h} v_k \cdot \e^{-\i \omega k/N}
	\Big|^2  \nonumber\\
	&  \leq (2N\pi)^{-2d} 
		\cdot \Big(\sum\limits_{k\in \Omega_h} \underbrace{\e^{-\i \omega k/N}\cdot \e^{\i \omega k/N}}_{=1}\Big)\Big(\sum\limits_{k\in \Omega_h} v_k^2 \Big)
		\nonumber\\
	&\leq (2\pi)^{-2d} N^{-d} \left(C_\Omega\cdot N^{d}\right) \cdot N^{-d}\sum\limits_{k\in\Omega_h} v_k^2 \nonumber\\
	&= \frac{C_\Omega}{(2\pi)^{2d}}\Lt{v_h}^2\label{eq:pointwise_bound_vhat}
	\end{align}
	
	where Cauchy-Schwarz inequality was used in the second line
	and \cref{eq:sinc_L2_l2_identity} for the last line.
	With this,
	proceed exactly as in the first proof the authors provide for 
	\cite[Theorem 3.7]{Covi2020}:
	For $\eps > 0$, we have
	\begin{align}
		\frac{1}{(2\pi)^d} \Lt{v_h}^2 &= \int_{B_\eps} \abs{\hat{v_h}(\omega)}^2\dom
				+ \int_{B_\eps^c} \abs{\hat{v_h}(\omega)}^2\dom \\
				&\leq \eps^{d}\abs{B_1(0)} \frac{C_\Omega}{(2\pi)^{2d}} \Lt{v_h}^2
						+ \eps^{-2s} \Hs{v_h}^{2}.
	\end{align}
	Rearranging  and choosing an appropriate $\eps$ yields the desired result.
\end{proof}

Besides the Poincar\'e inequality, we also have the following \emph{inverse estimate}. 
\begin{lemma}
	\label{lemma:inverse_estimate}
	Let $v_h\in\V_h(\Rd)$, then there is a constant $C>0$ independent of $h$, such that
	\begin{equation}
		\label{eq:Hs_bound_L2_sinc}
		\Hs{v_h} \leq C N^{s}\Lt{v_h} .
	\end{equation}
\end{lemma}
\begin{proof}
	This is a direct consequence of the fact the Fourier transformations of the functions
	in $\V_h$ have bounded support, noting
	\begin{align}
		\Hs{v_h}^2 &= (2\pi)^d\int_{D_N} \abs{\omega}^{2s}\abs{\hat v_h}^2\dom \nonumber\\
			&\leq (2\pi)^d\sup\limits_{\omega\in D_N}\abs{\omega}^{2s}\cdot\int_{D_N} \abs{\hat v_h}^2\dom
	    =(2\pi)^d \left(\sqrt{d}N\pi\right)^{2s} \Lt{v_h}^2.
	\end{align}
	Taking the square-root on both sides gives the desired result.
\end{proof}

\subsection{\texorpdfstring{$\sinc$}{sinc}-Galerkin method}
As in \cref{sub:weak_formulation_bvp}, we can define a Galerkin method on the space
$\V_h$ and aim to find a solution $u_h \in\V_h(\Omega)$ which fulfills
\begin{equation}
	\label{eq:discrete_galerkin_equation}
	a(u_h, v_h) = F_h(v_h)\quad\forall v_h\in \V_h.
\end{equation}
The bilinear form $a(\cdot,\cdot)$ is defined as in \cref{eq:def_a_on_V}
for the elements of $\V$. If we assemble the system matrix for the $\sinc$-Galerkin
method, we see that for two basis functions $\varphi^N_k, \varphi^N_j$, we have
\begin{align}
	a(\varphi^N_k, \varphi^N_j) &= 	 \int\limits_\Rd \halfflap \varphi^N_k\cdot \halfflap \varphi^N_j\dx \nonumber\\
			&=  (2\pi)^d \int_{\R^d} |\omega|^{2s} \hat\varphi^N_k(\omega) \left(\hat\varphi^N_j(\omega)\right)^*\dom \nonumber\\
			&= (2\pi)^d (2\pi N)^{-d} (2\pi N)^{-d} \int_{D_N} |\omega|^{2s}\e^{-\i \omega (k-j)/N}\dom \nonumber\\
			&= N^{-d} \Phi^N(k-j)
\end{align}
and note that -- up to scaling -- these are exactly the entries of the discrete operator $\Phi^N$ 
the authors obtained in \cite{ADS2021} for the discrete operator $\Phi^N$.
\emph{In other words, we can formulate the Galerkin method equivalently to the collocation
method in the sense that we obtain the same system matrix up to a 
scaling factor $N^{-d}$}. As a consequence, we can use the techniques shown in
\cite{ADS2021} to assemble the system matrix and solve the resulting system
efficiently.

The linear form $F_h$ requires some further discussion.
Problem \cref{eq:bvp} is posed on a bounded domain $\Omega\subset\Rd$
and extended with Dirichlet boundary conditions on $\Omega^c$. To be able to use
Fourier methods, one would however need to treat it as a problem on the
entire $\Rd$. In principle, the best choice of the right hand side would be given as follows.
For the solution $u$ of \cref{eq:bvp}
we take $\hat f$ as the Fourier transformation of its fractional
Laplacian, i.e., 
\begin{equation}
	\label{eq:def_fhat}
	\hat f(\omega) \coloneqq |\omega|^{2s} \hat u(\omega).
\end{equation}
Then, we set the right hand side in \cref{eq:discrete_galerkin_equation} to
\begin{equation}
	\label{eq:Fh_proj_f_varphiNk}
	F_h(\varphi^N_k) = (f, \varphi^N_k)_{L^2(\Rd)} = N^{-d} \int_{D_N} \hat f \e^{\i\omega k/N}\dom.
\end{equation}
This is in principle well defined, but not computable. To implement \cref{eq:Fh_proj_f_varphiNk}, we
would need to know the analytical solution $u$, as otherwise we only know the values of the inverse Fourier transform of $\hat f$ inside of $\Omega$ through the given right hand side $f$. 

Defining the right-hand side this way is in line with extending $f$ with the
nonlocal derivative of $u$ outside of $\Omega$, defined as
\begin{equation}
	\Ns u(x) =  C(d,s)\int_\Omega \frac{u(x) - u(y)}{\abs{x-y}^{d+2s}}\dy
	= -C(d,s)\int_\Omega \frac{u(y)}{\abs{x-y}^{d+2s}}\dy, \quad \text{for } x\in \R^d\setminus \bar\Omega,
		u\in\tildeHsOm
\end{equation}
where $C(d,s)$ is the constant from \cref{eq:integral_def_flap}.  As we are solving the Dirichlet problem, with $u$ necessarily vanishing outside of $\Omega$ for our approximation to be valid, i.e., for $u\in \tildeHsOm$, the term
$\Ns u$ is simply the fractional Laplacian of $u$ evaluated outside of $\Omega$. 
For $u\in \tildeHsOm, v\in H^s(\Rd)$, the bilinear form $a(\cdot,\cdot)$ can be written as
\begin{equation}
	a(u, v) = \frac{C(d,s)}{2}  \iint\limits_{(\Rd\times\Rd)\setminus (\Omega^c\times\Omega^c)} \frac{(u(x)-u(y))\,(v(x)-v(y))}{\abs{x-y}^{d+2s}}\dx\dy,
\end{equation}
which, using the integration by parts formula \cite{Dipierro2017}, can be written as
\begin{equation}
	a(u,v) = \int_\Omega v(x) \flap u(x)\dx + \int_{\Omega^c} v(x) \Ns u(x)\dx
		= \int_{\Rd} v(x) \flap u(x)\dx ,
\end{equation}
Since $f$ is defined in $\Omega$, an extension of $f$ to 
outside $\Omega$ is given by $\mathcal{N}_s u(x)$.

Notice that, if the right-hand side $\hat f$ defined in \cref{eq:def_fhat} 
were to decay sufficiently fast, i.e.,  if
\begin{equation}
    \label{eq:def_f_bar}
	\bar f(x) \coloneqq \IFT \hat f = \begin{cases}
		f(x)&\text{if }x\in\Omega \\
		\Ns u(x)&\text{otherwise}
	\end{cases}
\end{equation}
was globally sufficiently smooth, the approximation 
\begin{equation}
	(\hat f, \hat\varphi^N_k)_{L^2} 
	= N^{-d}\int_{D_N} \hat f(\omega) \e^{\i\omega k/N}\dom \approx N^{-d}\int_{\Rd} \hat f(\omega) \e^{\i\omega k/N}\dom
	= N^{-d} f(k/N),
\end{equation}
which requires only the known values of $f$ inside of $\Omega$,
is valid. However, we can not expect such global smoothness of $\bar f$.

The above remark shows that some more effort is required to obtain a usable Galerkin method for our fractional PDE.
We need to obtain a suitable approximation $f_h \in\V_h(\Rd)$ of $\bar f$,
so that we can define the linear form
$F_h$ through
\begin{equation}
	F_h(\varphi^N_k) = (f_h, \varphi^N_k)_{L^2(\Rd)}.
\end{equation}
This can be computed with no additional effort once the coefficients $f_j, j\in\Z^d$
of $f_h = \sum_{j\in\Zd}f_j\varphi^N_j(x)$ are known, since
\begin{equation}
	F_h(\varphi^N_k) = (f_h, \varphi^N_k)_{L^2(\Rd)} 
		= \left(\sum_{j\in\Zd}f_j\varphi^N_j(x), \varphi^N_k\right)_{L^2(\Rd)}
		= \sum\limits_{j\in\Z^d}
			f_j \cdot \underbrace{\left(\varphi^N_k, \varphi^N_j\right)_{L^2(\Rd)}}_{=N^{-d}\delta_{j,k}}
		= N^{-d} f_k.
\end{equation}
In \cref{sub:approx_rhs} we show
how a suitable discrete right-hand side $f_h \in \V_h(\Rd)$ can be computed.

\subsection{Abstract error estimate}
\label{sub:abstract_error_estimate}
Notice that $\V_h(\Omega) \not\subset\V$
and even $\V_h(\Omega)\cap\V = \{0\}$ because the elements of $\V_h$ have unbounded support 
in contrast to the elements of $\V$ whose support is in $\overline \Omega$.
 Therefore,  C\'ea's theorem is not applicable and we follow \cite[Chapter 10]{Brenner2007} to obtain an abstract error 
estimate of Strang type.
\begin{theorem}[{\cite[Lemma 10.1.1]{Brenner2007}}]
	\label{thm:abstract_error}
	Let $V$ and $V_h$ be subspaces of a Hilbert space $H$. 
	Assume that $a(\cdot, \cdot)$ is a continuous bilinear form
	on $H$ which is coercive on $V_h$, with continuity and coercivity
	constants $C$ and $\gamma$, respectively. Let $u\in V$ and $u_h\in V_h$ solve
	\begin{equation}
		\label{eq:strang_continuous_equation}
		a(u, v) = F(v)\quad\forall v\in V
	\end{equation}
	and
	\begin{equation}
		\label{eq:strang_discrete_equation}
		a(u_h, v_h) = F_h(v_h)\quad\forall v\in V_h .
	\end{equation}
	Then
	\begin{equation}
	    \label{eq:full_strang_error_equation}
		\norm{u-u_h}_H \leq \left( 1 + \frac{C}{\gamma}  \right) \inf\limits_{v_h\in V_h} \norm{u - v_h}_H
			+ \frac{1}{\gamma}  \sup\limits_{w_h\in V_h\setminus\{0\}} \frac{\abs{a(u-u_h, w_h)}}{\norm{w_h}_H}.
	\end{equation}
\end{theorem}
 Proofs can be found, e.g., in \cite[Lemma 27.15]{Ern2021} or \cite[Lemma 10.1.1]{Brenner2007}.  As usual, we refer to the first
term on the right hand side as the \emph{approximation error} and to
the second term as the \emph{consistency error}.
As the underlying Hilbert space $H$, we choose $H \coloneqq H^{s}(\Rd)$. The
spaces $V$ and $V_h$ are as defined above,
namely $V = \V$ and $V_h = \V_h(\Omega)$. The bilinear form $a(\cdot, \cdot)$
is obviously continuous on $\V$ and coercive on $\V_h(\Omega)$ as a direct consequence of \Cref{lemma:poincare}.

Using that $a(u - u_h,v) = (\flap (u-u_h), v)_{L^2(\Rd)}$,
$\flap u = \bar f$ on $\Rd$, see \cref{eq:def_f_bar,eq:strang_discrete_equation},
it immediately follows that the numerator of the consistency term can be written as
\begin{equation}
	a(u-u_h, w_h) = (\bar f - f_h, w_h)_{L^2(\Rd)}.
\end{equation}
As $w_h\in\V_h(\Omega)$, $(\varphi^N_j, w_h)_{L^2(\Rd)} = 0$ if $j/N\not\in\Omega$.
Therefore, it is sufficient that the coefficients of the approximation $f_h$ match the coefficients
of the right-hand side of our linear system of equations in $\Omega$ and the coefficients outside of $\Omega$ are immaterial. Given such
a $\sinc$-approximation $f_h$ of $\bar f$  (which is yet to be constructed) we
continue the above computation  via
\begin{equation}
a(u-u_h, w_h) = \langle \overline f - f_h, w_h \rangle_{H^{-s}(\mathbb{R}^d) \times H^{s}(\mathbb{R}^d)} \le \| \overline f - f_h \|_{H^{-s}(\mathbb{R}^d)} \| w_h\|_{H^{s}(\mathbb{R}^d)}.
\end{equation}
with $\norm{\bar f}_{H^{-s}(\Rd)} = \norm{(1+\abs{\omega})^{-s/2})\hat f}_{L^2(\Rd)}$ see \cite{Dinezza2012}.
With this, bounding the consistency error
reduces to finding a $\sinc$-function $f_h$ which approximates $\bar f$ in the
$\Hms{\,\cdot\,}$-norm and whose coefficients we can actually compute for $k\in\Omega_h$.

In summary, finding the error reduces to 
\begin{itemize}
\item[(i)] showing that the true solution can
actually be approximated in the discrete space and
\item[(ii)] finding an
approximation of the right-hand side in the $\sinc$-basis.
\end{itemize}

\section{Error bounds}
In this section, we bound the individual terms in \cref{eq:full_strang_error_equation} to obtain optimal rates of convergence in the energy norm.
\label{sec:approximation_errors}
\subsection{\texorpdfstring{$\sinc$-Interpolation in $\widetilde H^{s}(\Omega)$}{sinc-Interpolation in tildeHs(Omega)}}
\label{sub:interpolation_errors_u}
We need to find a function $v_h\in\V_h(\Omega)$
that suitably approximates $u$ solving \cref{eq:bvp} in $\Hsnorm{\,\cdot\,}$.
We use this element to bound the approximation error in \Cref{thm:abstract_error}.
Our overall strategy is
be as follows. We compute a smooth approximation of $u$ via mollification, interpolate
this smooth approximation using the interpolation operator $\PiN$ defined in 
\eqref{eq:def_interp_operator} and restrict
the thus obtained element of $\V_h(\Rd)$ to an element of $\V_h(\Omega)$ via the
restriction operator defined in \cref{eq:def_discrete_restriction}.
Subsequently, we establish the optimal rate of convergence.

For a standard mollifier $\mollif(x)$, with $\supp\mollif\subset B_1(0)$,
let $\mollif_\eps(x) \coloneqq \eps^{-d} \mollif(x/\eps)$ and define the smoothing operation
through convolution, i.e., for a solution $u$ of \cref{eq:bvp}, we set
$u_\eps \coloneqq \mollif_\eps * u$ and immediately have $\FT\left\{\mollif_\eps * u\right\} = \hat{\mollif}_\eps\hat u$
where $\hat {\mollif}_\eps(\omega) = \hat\mollif(\eps\omega)$.

Given the results in \cref{sub:summary_regularity_results}, we may assume that $u$ has some additional regularity above $H^s(\R^d)$
so take $u \in \tildeHtOm \subset \tildeHsOm$ for some $t > s$.
Additionally, we may assume $\abs{u(x)} \leq \cbound \dist(x,\delOm)^s$
near the boundary
$\delOm$ of $\Omega$. See especially \Cref{rem:def_cbound_csob}.

\begin{lemma}
	\label{lemma:sinc_interp_smoothing}
	Let $u\in\widetilde{H}^t(\Omega)$ for $t > s$. Then
	\begin{enumerate}[(i)]
		\item\label{interp_smooth_i} For $\eps > 0$ and $t-s < 2$, we have
			\[
				\Hs{u-u_\eps} \leq \normalfont{\ci} \eps^{t-s} \Ht{u}
			\]
		\item\label{interp_smooth_ii} Choose $\eps = 1/N$, then
			\[
				\Hs{u_\eps - \PiN{u_\eps}} \leq \normalfont{\cii} \eps^{t-s}\Ht{u}.
			\]
	\end{enumerate}
	Where $\normalfont{\ci,\cii}$ are constants that depend only on
	$d, s$ and $\eta$, but not on $\Omega$.
\end{lemma}
\begin{proof}
	(\ref{interp_smooth_i}) 
	 The same result, but with different arguments, can be found in \cite[Theorem 2]{Heltai2020}.
	For the readers convenience, we include a simple proof here. 
	A direct computations shows that
 \begin{align*}
	 \Hs{u-u_\eps}^2 & = \int_{\mathbb{R}^d} \abs{\omega}^{2s} (1-\hat\eta(\eps \omega))^{2s} \abs{\hat u(\omega)}^{2} d\omega  =  \eps^{2t-2s} \int_{\mathbb{R}^d} \abs{\eps \omega}^{2s-2t} (1-\hat\eta(\eps \omega))^{2s} \abs{\omega}^{2t} \abs{\hat u(\omega)}^{2} d\omega \\
	 & \le   \eps^{2t-2s} \left(\sup_\omega  \abs{\eps\omega}^{2s-2t}(1-\hat\eta(\eps \omega))^{2s}\right) \int_{\mathbb{R}^d} \abs{\omega}^{2t} \abs{\hat u(\omega)}^{2} d\omega \\
& = C_{(i)} \eps^{2t-2s} | u |_{H^t(\mathbb{R}^d)}^2 ,
\end{align*}
	
	taking the square root on both sides yields the desired result provided that the constant $\ci$ is finite.
	For $\abs{\omega}\longrightarrow\infty$, this is clear when $t > s$, so we only need to
	show that this is also the case near $0$. This holds true under the assumption that $t - s < 2$. Indeed, as $\mollif$ is a smooth and symmetric mollifier, $\hat \mollif$ is smooth and
	real and, thus, has a Taylor expansion
	\begin{equation}
	    \hat\mollif(\omega) = \hat\mollif(0) + \nabla_\omega\hat\mollif(0)\cdot \omega + \mathcal{O}(\abs{\omega}^2)
	\end{equation}
	around $0$. The partial derivatives of $\hat\mollif$ are given by
	\begin{equation}
	    \frac{\partial}{\partial\omega_j} \hat\mollif(\omega) =
	        -\i\int_\Rd x_j \mollif(x) \e^{-\i\omega x}\dx \overset{\omega=0}{=}
	        -\i\int_\Rd x_j \mollif(x) \dx = 0.
	\end{equation}
	The last equality holds because $\eta$ symmetric, thus $\hat\eta$ is symmetric.
	As this holds for $j = 1,\ldots,d$, we have $\nabla_\omega\hat\mollif(0)\cdot\omega = 0 \;\forall\omega$. This, along with $\hat\mollif(0) = \int_\Rd\mollif(x)\dx = 1$ allows
	us to show that 
	\begin{equation}
	    1 - \hat\mollif(\omega) =
	    1 - \big(\hat\mollif(0) + \nabla_\omega\hat\mollif(0)\cdot \omega + \mathcal{O}(\abs{\omega}^2)\big) 
			= \mathcal{O}(\abs{\omega}^2),
	\end{equation}
	and we obtain
	\begin{equation}
	    |\omega|^{2s-2t} (1-\hat\mollif(\omega))^2 \leq C |\omega|^{2s-2t} \abs{\omega}^4
	    < \infty\;\Leftrightarrow\; t-s < 2
	\end{equation}
	around $\omega = 0$.

	(\ref{interp_smooth_ii})
 Next, using $\PiN u_\eps \in \V_h(\Rd) \Rightarrow \supp{\widehat{\PiN u_\eps}} \subseteq D_N$, we obtain
	\begin{equation}
		\label{eq:interp_term_split}
		\Hs{u_\eps - \PiN u_\eps}^2 = 
		\underbrace{\int_{D_N} |\omega|^{2s} |\hat{u}_\eps - \widehat{\PiN{u_\eps}}|^2\dom}_{\text{(I)}}
			+ \underbrace{\int_{D_N^c} |\omega|^{2s}|\hat{u}_\eps|^2\dom}_{\text{(II)}}.
	\end{equation}
	We begin with the second-term and obtain
	\begin{align}
		\text{(II)} = \int_{D_N^c} |\omega|^{2s}|\hat{u}_\eps|^2\dom &\leq
				\int_{D_N^c} |\omega|^{2s} \abs{\hat\mollif(\eps\omega)}^2\abs{\hat{u}}^2\dom
			\nonumber\\
			&\leq \underbrace{\sup\limits_{\omega\in D_N^c} \abs{\hat\mollif(\eps\omega)}^2}_{\leq 1}
				\cdot \int_{D_N^c} |\omega|^{2s} \abs{\hat{u}}^2\dom , \label{eq:interp_err_term_I}
	\end{align}
where we have used the fact that 
 $\sup_{\omega\in\Rd}\hat\mollif(\omega) = 1$ is attained for $\omega = 0$ .	
Using $\hat\varphi^N_k(\omega) = (2N\pi)^{-d}\e^{-\i\omega k/N}$ in $D_N$, which is just the definition
of the basis functions, and the Poisson summation formula, we obtain 
\begin{equation}
	\hat{u}_\eps(\omega) - \widehat{\PiN{u}_\eps} =
	\hat{u}_\eps(\omega) - \frac{1}{(2N\pi)^d}\sum\limits_{k\in\Zd} u_\eps(k/N) \e^{-\i \omega k/N}
		= -\sum\limits_{k\in\Zd\setminus\{0\}}\hat u_\eps(2N\pi k + \omega)
\end{equation}
for $\omega\in D_N$. We plug this into \eqref{eq:interp_term_split} and obtain
\begin{align}
    \text{(I)} = \int_{D_N} |\omega|^{2s} |\hat{u}_\eps - \widehat{\PiN{u_\eps}}|^2\dom
        &= \int_{D_N} 
        |\omega|^{2s} \bigg| \sum\limits_{k\in\Zd\setminus\{0\}}\hat u_\eps(2N\pi k + \omega) \bigg|^2
        \dom \\
        &= (2N\pi)^{d+2s} \int\limits_{[-1/2,1,2]^d}|\omega|^{2s} \bigg|\sum\limits_{k\neq 0}\hat u_\eps (2N\pi (k + \omega))\bigg|^2\dom\label{eq:interp_term_split_II_ii}
\end{align}
where the second equality is obtained by a linear integral transformation. 
To verify that the Poisson summation formula  is applicable to $u_\eps$, we have to verify that both,
$u_\eps$ and $\hat u_\eps$ fulfill the decay condition required in \cref{thm:poisson_summation}.
For $u_\eps$, this is clear as $u_\eps$ is smooth and has compact support. For 
$\hat u_\eps = \hat u \cdot \hat\mollif_\eps$, this holds as $u \in L^1(\Rd)$, thus as
a consequence of the Lemma of Riemann-Lebesgue, 
$\hat u$ is continuous and $\hat u(\omega) \rightarrow 0$ for $\abs{\omega}\rightarrow\infty$.
Further, $\mollif_\eps$ is smooth, thus $\hat\mollif_\eps$ decays exponentially. Combining
the two arguments, the required decay of $\hat u_\eps$ follows.

We continue, using the Cauchy-Schwarz inequality,
\begin{align}
	&\bigg|\sum\limits_{k\neq 0}\hat u_\eps (2N\pi (k + \omega))\bigg|^2 = 
    \bigg|\sum\limits_{k\neq 0}\hat u (2N\pi (k + \omega)) \cdot \hat \mollif (\eps2N\pi (k + \omega)) 
           \cdot \frac{(2N\pi (k + \omega))^s}{(2N\pi (k + \omega))^s} \bigg|^2 \\
	&  \qquad\leq \sum\limits_{k\neq 0} \hat \mollif (2\pi (k + \omega))^2(2N\pi (k + \omega))^{-2s}
            \cdot
						\sum\limits_{k\neq 0} \abs{\hat u (2N\pi (k + \omega))}^2 (2N\pi (k + \omega))^{2s} \\
	&	\qquad= (2N\pi)^{-2s}
    \underbrace{\sum\limits_{k\neq 0} \hat \mollif (2\pi(k + \omega))^2((k + \omega))^{-2s}}_{\leq c_\mollif^2}
            \cdot
						\sum\limits_{k\neq 0} \abs{\hat u (2N\pi (k + \omega))}^2 (2N\pi (k + \omega))^{2s}
            \label{eq:cs_fourier_sum}.
\end{align}
In the second line, we have used $\eps = 1/N$. The constant $c_\mollif$ defined  by
\begin{equation}
	c_\mollif^2 \coloneqq \sup\limits_{\omega\in [-1/2,1/2]^d}
		\sum\limits_{k\neq 0} \hat\mollif^2(2\pi(k+\omega)) |(k+\omega)|^{-2s}
\end{equation}
is finite as $\hat\mollif$ decays exponentially and does not depend on $N$.
By plugging this result into \cref{eq:interp_term_split_II_ii},
we obtain
\begin{align}
    \text{(I)} &\leq
    c_\mollif^2 \cdot (2N\pi)^{d}\int\limits_{[-1/2,1,2]^d}
		\underbrace{|\omega|^{2s}}_{\leq (\sqrt{d}/2)^{2s}} 
		\sum\limits_{k\neq 0}\abs{\hat u (2N\pi (k + \omega))}^2 (2N\pi (k + \omega))^{2s} \dom \\
    &\leq c_\mollif^2 (\sqrt{d}/2)^{2s} \cdot \int\limits_{[-N\pi,N\pi]^d} \sum\limits_{k\neq 0}
		\abs{\hat u (2N\pi k + \omega)}^2 |2N\pi k+\omega|^{2s} \dom\\
   & = c_\mollif^2\cdot (\sqrt{d}/2)^{2s}\cdot \int\limits_{D_N^c} |\omega|^{2s}\abs{\hat u(\omega)}^2\dom,
		\label{eq:interp_err_term_II}
\end{align}
where the exchange of summation and integration in the last equation is justified using the monotone convergence theorem.
Now, combining \cref{eq:interp_err_term_I}, \cref{eq:interp_err_term_II}  and  \cref{eq:interp_term_split}
we obtain
\begin{equation}
	\label{eq:interp_err_term_combined_i}
	\Hs{u_\eps - \PiN u_\eps}^2 \leq \text{(I)} + \text{(II)} \leq (c_\mollif^2\cdot (\sqrt{d}/2)^{2s} + 1)
		\int\limits_{D_N^c} |\omega|^{2s}{\hat u}^2(\omega)\dom.
\end{equation}
By our assumption that  $u\in H^t(\Rd)$, we obtain an error decay rate, viz
\begin{equation}
\int_{D_N^c} |\omega|^{2s}{\hat u}^2(\omega)\dom \leq
	(d(\pi N)^2)^{s-t} \int\limits_{D_N^c} |\omega|^{2t}{\hat u}^2(\omega)\dom
	\leq(d\pi^2)^{s-t}\cdot N^{2(s-t)}\cdot \Ht{u}^2,
\end{equation}
which, substituted into \cref{eq:interp_err_term_combined_i}, yields
\begin{equation}
	\Hs{u_\eps - \PiN u_\eps} \leq 
	\underbrace{\left( (1+c_\mollif^2 (\sqrt{d}/2)^{2s}) (d\pi^2)^{s-t} \right)^{1/2}}_{\eqqcolon \cii}
		\cdot N^{s-t}\cdot \Ht{u},
\end{equation}
as promised.

The change of the sum and the integral in \cref{eq:interp_err_term_II}  needs further justification.
We want  to use the monotone convergence theorem  to show that
\begin{equation}
	\int\limits_{ D_N} \sum\limits_{k\neq 0}\abs{\hat u(2N\pi k + \omega)}^2 |2N\pi k + \omega|^{2s}\dom
		= \int\limits_{D_N^c} |\omega|^{2s}\abs{\hat u(\omega)}^2\dom.
\end{equation}
For $m \geq 1$, we denote $mD_N = [-mN\pi, mN\pi]^d$. Let further denote 
${\I_m = \{-m,\ldots,m\}^d}$. Then
\begin{align}
	\int\limits_{ D_N^c} |\omega|^{2s}\abs{\hat u(\omega)}^2\dom &=
		\int\limits_{m\cdot D_N\setminus D_N} |\omega|^{2s}\abs{\hat u(\omega)}^2\dom
			+ \underbrace{\int\limits_{(m\cdot D_N)^c} |\omega|^{2s}\abs{\hat u(\omega)}^2\dom}_{\coloneqq R(m)} \\
		&= \sum\limits_{k \in \I_m\setminus\{0\}}
			\int\limits_{D_N} |\omega+2mN\pi|^{2s}\abs{\hat u(\omega+2mN\pi)}^2\dom +\;R(m)\\
		&= \int\limits_{D_N} \sum\limits_{k \in \I_m\setminus\{0\}}
			|\omega+2mN\pi|^{2s}\abs{\hat u(\omega+2mN\pi)}^2\dom + \;R(m).
\end{align}
The sequence of functions $(F_m)_{m\in\N}$ given by
\begin{equation}
	F_m(\omega) = \sum\limits_{k \in \I_m\setminus\{0\}}
			|\omega+2mN\pi|^{2s}\abs{\hat u(\omega+2mN\pi)}^2
\end{equation}
is pointwise monotonous increasing as $\abs{\hat u(\xi)}^2 |\xi|^{2s} \geq 0$. The integral
of the $F_m$ is bounded because
\begin{equation}
	\int\limits_{D_N} F_m\dom = 
	\int\limits_{D_N} \sum\limits_{k \in \I_m\setminus\{0\}}
			|\omega+2mN\pi|^{2s}\abs{\hat u(\omega+2mN\pi)}^2\dom < \int\limits_{D_N^c}\abs{\hat u(\omega)}^2|\omega|^{2s}\dom
				< \infty.
\end{equation}
 Therefore, we can apply the monotone convergence theorem and the proof is complete.
\end{proof}
\begin{remark}
	The above lemma holds for fractional exponent $0$ as well. By considering the $\Lt{\,\cdot\,}$-norm
	as the $\abs{\,\cdot\,}_{H^0(\Rd)}$-seminorm, we obtain estimates with respect to 
	$\Hsnorm{\,\cdot\,}$, i.e., 
	\begin{equation}
		\Hsnorm{u - u_\eps}^2 =  \Lt{u - u_\eps}^2 + \Hs{u - u_\eps}^2
			\leq N^{-t} + C N^{s-t} \leq C N^{s-t}
	\end{equation}
	and
	\begin{equation}
		\Hsnorm{u_\eps - \PiN u_\eps}^2 =  \Lt{u_\eps - \PiN u_\eps}^2 + \Hs{u_\eps - \PiN u_\eps}^2
			\leq N^{-t} + C N^{s-t} \leq C N^{s-t}.
	\end{equation}
\end{remark}

The problem with the above approximation procedure is that the smoothing operation
enlarges the support of the function, thus, in general, ${\PiN u_\eps\not\in\V_h(\Omega)}$.
However, using the fact that the growth of solutions to \cref{eq:bvp} is bounded 
by a power of the distance to the boundary of $\Omega$, see \cref{eq:bound_u_boundary}, we can prove the following lemma.
\begin{lemma}
	\label{lemma:sinc_interp_coeff_drop}
	Let $u\in\tildeHsOm$ be such that $\abs{u(x)} \leq \cbound\dist(x,\delOm)^s$ for a positive constant 
	$\cbound$.
	For $x\in\Rd$, let $B_h(x)$ be the ball around $x$ with radius $h$. Assume that $\Omega$ is such that only $\mathcal O(N^{d-1})$-many points of the grid $h\Z^d$ are located within
	the strip of width $h$ around $\Omega$, namely
	\begin{equation} \label{eq:boundarypoints_bound}
		\abs{\{((\Omega + B_h)\setminus\Omega) \cap h\Zd\}} \leq D\cdot N^{d-1} 
	\end{equation}
	where
	$\Omega + B_h \coloneqq \left\{ x + y \,\big|\, x\in\Omega, y\in B_h(0)\right\}$
	and $D = D(\Omega)$ is a positive constant that does
	not depend on $h$. Then, for $\eps = h = 1/N$,
	\begin{equation}
		\Hsnorm{\PiN u_\eps - \resOm{\PiN u_\eps}} \leq \ciii \cdot \cbound\cdot h^{1/2}
	\end{equation}
\end{lemma}
where $\ciii$ is a positive constant that depends on $D(\Omega)$, $s$ and $d$.
\begin{proof}
    Let $x_k \coloneqq k/N\in (\Omega + B_h)\setminus \Omega$. As for
    $x\in\Omega \cap B_h(x_k) \neq \emptyset$, it holds that
    \begin{equation}
        \dist(x,\delOm) \leq \abs{x_k - x} \leq h
    \end{equation}
    and it follows that
    \begin{equation}
			\sup\limits_{x\in\Omega\cap B_h(x_k)}\abs{u(x)} \leq \cbound\dist(x,\delOm)^s \leq \cbound h^s.
    \end{equation}
    Therefore
    \begin{align}
        u_\eps(x_k) &= \int_{B_\eps(x_k)} \eta_\eps(x-x_k) u(x)\dx \nonumber\\
            &\leq \sup\limits_{x\in B_\eps(x_k)} u(x) \cdot 
            \underbrace{\int_{B_\eps(x_k)}\eta_\eps(x-x_k)\dx}_{=1}
            \leq \cbound\, \eps^s = \cbound\, h^s,
    \end{align}
    noting that we chose $h=\eps$.
	Now, let $v_h \coloneqq \PiN{u_\eps}$ and $v_h^\Omega \coloneqq
	\resOm{\PiN{u_\eps}}$.  Recall that $\resOm{\cdot}$ does not actually
	restrict the support of the function, but only the index set in the sum
	as defined and explained in \cref{eq:def_discrete_restriction}. 
	Using this, first, we compute the desired bound in the $L^2(\Rd)$-norm, i.e.,
	\begin{align}
		\Lt{v_h^\Omega - v_h}^2 &= N^{-d}\sum\limits_{k\in\Zd} (v_h(k/N) - v_h^\Omega(k/N))^2
			= N^{-d}\sum\limits_{k\in\Zd\setminus\Omega_h} v_h(k/N)^2 \nonumber\\
			&\leq N^{-d}\underbrace{\Big(\text{\#gridpoints in }(\Omega + B_h)\setminus\Omega\Big)}_{\leq D N^{d-1}}
				\cdot \underbrace{ \max_{k\in\Zd\setminus\Omega_h} \left(v_h(k/N)\right)^2}_{\leq \cbound^2 h^{2s}} \nonumber\\
			&= \cbound^2\cdot D N^{-1}\cdot h^{2s} = \cbound^2 \cdot D \cdot h^{1+2s}
	\end{align}
	Using \Cref{lemma:inverse_estimate} and the above estimate, we obtain that 
	\begin{align}
		\Hsnorm{v_h^\Omega - v_h}^2 &= \Lt{v_h^\Omega - v_h}^2 + \Hs{v_h^\Omega - v_h}^2 \nonumber\\
			&\leq (1+C\cdot N^{2s})\Lt{v_h^\Omega - v_h}^2 \nonumber\\
			&\leq  (1+C\cdot N^{2s}) \cdot\cbound^2 \cdot D \cdot h^{1+2s} \nonumber\\
			&\leq \ciii^2 \cdot \cbound^2\cdot h
	\end{align}
	where in the second line, $C = C(d,s)$ is the constant from \cref{lemma:inverse_estimate}.
	The proof is complete.
\end{proof}

Combining \Cref{lemma:sinc_interp_smoothing,lemma:sinc_interp_coeff_drop},
we can show the following bound on the approximation error.
\begin{proposition}
	\label{prop:bound_approximation_error}
	Let $\Omega$ be a bounded Lipschitz domain that fulfills the exterior ball condition
	and $f$ as in \Cref{thm:sobolev_reg_u_lipschitz}. Then, there exists a constant
	$\capprox(\Omega, f) > 0$ such that
	\begin{equation}
		\inf\limits_{v_h\in\V_h(\Omega)}\Hsnorm{u - v_h} \leq \frac{\capprox(\Omega, f)}{\theta}  h^{1/2-\theta}, \quad \text{for all $\theta>0$}.
	\end{equation}
\end{proposition}
\begin{proof}  We first note that \eqref{eq:boundarypoints_bound} is satisfied for any bounded Lipschitz domain. To see this, one may use a covering of $(\Omega + B_h)\setminus \Omega$ by balls of radius $5h$ centered on $\partial \Omega$ and observe that no more than $\frac{C}{h^{d-1}}$ such balls are necessary, where $C$ depends on the Lipschitz constant of the domain and the perimeter of $\partial \Omega$. The number of points on $h\Z^d$ contained in each such ball is uniformly bounded.  Now,  by using $t = s+1/2-\theta$ in \Cref{lemma:sinc_interp_smoothing,lemma:sinc_interp_coeff_drop}
	and \Cref{thm:sobolev_reg_u_lipschitz}, we obtain 
	\begin{align} \inf\limits_{v_h\in\V_h(\Omega)}\Hs{u - v_h} &\leq \Hs{u - \resOm{\PiN{u_\eps}}}  \nonumber\\
		&\leq \Hs{u - u_\eps} + \Hs{u_\eps - \PiN u_\eps} + \Hs{\PiN u_\eps - \resOm{\PiN{u_\eps}}} \nonumber\\
		&\leq (\ci + \cii) h^{t-s}\Ht{u} + \ciii\cdot\cbound\cdot h^{1/2} \nonumber\\
		&\leq (\ci + \cii) h^{1/2-\theta}\frac{\csob(\Omega, f)}{\theta}  + \ciii\cdot\cbound\cdot h^{1/2}
	\end{align} 
	which shows the desired result.
\end{proof}
If $f$ is less reguar but $\Omega$ is smooth, we can derive the following 
estimate using \Cref{thm:sobolev_reg_u_smooth}.
\begin{proposition}
	\label{prop:bound_approximation_error_2}
	Let $\Omega$ be a smooth domain that fulfills the exterior ball condition
	and let $f \in L^{\infty}(\Omega)$. Then, there exists a constant
	$\capprox(\Omega, f, d, s) > 0$ such that
	\begin{equation}
		\inf\limits_{v_h\in\V_h(\Omega)}\Hsnorm{u - v_h} \leq \capprox(\Omega, f, d, \alpha) h^{\alpha}
	\end{equation}
	where $\alpha = \min(s, \frac{1}{2}-\theta)$ for $\theta$ arbitrarly small.
\end{proposition}
\begin{proof}
	As $\Omega$ is bounded, $f\in L^{\infty}(\Omega)$ implies $f\in L^{2}(\Omega)$.
	Setting $r = 0$ in \Cref{thm:sobolev_reg_u_smooth}, we know that
	\[
		\Ht{u} \leq C(d,t)\norm{f}_{L^{2}(\Omega)}
	\]
	where 
	\[
		t = \begin{cases}
			2s&\text{if } s < 1/2 \\
			s + 1/2 - \theta&\text{otherwise}
		\end{cases}
	\]
	for all $\theta > 0$. With this, the proof is the same
	as in \Cref{prop:bound_approximation_error}.
\end{proof}

\subsection{Approximation of the right-hand side}
\label{sub:approx_rhs}
This section is devoted to the question how one can bound the consistency
term in the abstract error estimate that is \Cref{thm:abstract_error}. Connected to this question, we discuss how the right hand side of the discrete equation, see \cref{eq:formulation_collocation}, has to be chosen. 

Since we will -- in the following -- need to enlarge our domain $\Omega$ to be able to mollify the right hand side $f$ before evaluation on grid points, we require the regularity theorems to hold uniformly on such enlarged domains (i.e., $\cbound$ and $\csob$ defined in \cref{sub:summary_regularity_results} should remain bounded for slightly enlarged domains). This leads to the following assumption.
\begin{assumption} \label{ass:dom}
There is a Lipschitz constant $L>0$ and radii $r>0$, $R>0$ so that $\Omega+B_\rho$ are Lipschitz domains with Lipschitz constant no larger than $L$ and satisfy an exterior ball condition for balls of radius $r$, uniformly for $0\le\rho<R$.
\end{assumption}
It is easy to check that for typical domains like convex polygons or domains with smooth boundary satisfy this assumption. Further, see Theorem 1.1 in \cite{Belegradek2017} and the remark after the theorem to note that if $\Omega$ is Lipschitz and convex, then so is $\Omega + B_\rho$.

We first consider the case when $f\in L^\infty(\Omega)$. It is clear that a simple interpolation of $f$ will not be possible in this case. We proceed in  two  steps.
\begin{description}
\item[\textit{Step 1.}] Assuming we knew $\bar{f}$ from \cref{eq:def_f_bar} (i.e., the fractional Laplacian applied to the solution $u$ on the entire $\R^d$), we bound the error that is made by sinc-interpolating a mollified $\bar{f}$.
\item[\textit{Step 2.}] We show that the difference between $\bar{f}$ and the original $f$ continued onto a slightly enlarged domain can also be bounded. On this enlarged
	domain, the smoothing operation can be computed on the grid points in the original domain $\Omega$.
\end{description}
Afterwards, we discuss in which cases one can omit the onerous extension-and-mollification steps and directly interpolate a given right hand side $f$.

\textbf{Step 1} above is treated by the following lemma.
\begin{lemma}
	\label{lemma:sinc_interp_smoothing_f}
	Let $|\omega|^{2s}\hat u = \hat f = \FT\bar{f}$, and $u\in \widetilde H^t(\Omega)$, $t > s$, as before. For an appropriately chosen mollifier $\eta_\eps$ (see proof), set $\bar f_\eps \coloneqq \eta_\eps *\bar f$.
	Then,
	\begin{enumerate}[(i)]
		\item\label{interp_smooth_rhs_i} For $\eps > 0$, we have
			\begin{equation}
				\Hms{\bar f-\bar f_\eps} \leq \eps^{t-s} \Ht{u} \, .
			\end{equation}
		\item\label{interp_smooth_rhs_ii} For $\eps = 1/N$, we have
			\begin{equation}
				\Hms{f_\eps - \PiN{f_\eps}} \leq \eps^{t-s}\Ht{u} \, .
			\end{equation}
	\end{enumerate}
\end{lemma}
\begin{proof}
	(\ref{interp_smooth_rhs_i}) Using $\hat f = \abs{\omega}^{2s}\hat u$ and 
	$w(\omega)^{-s} \coloneqq\big(1+\abs{\omega}^{2}\big)^{-s} \leq \abs{\omega}^{-2s}$, the same proof as for part (i) in \Cref{lemma:sinc_interp_smoothing} applies.

	(\ref{interp_smooth_rhs_ii})
	Similarly as in the proof of the interpolation errors of the solution in \Cref{lemma:sinc_interp_smoothing}(ii), see
	\Cref{lemma:sinc_interp_smoothing}, we compute 
	\begin{equation}
		\Hs{f_\eps - \PiN f_\eps}^2 = 
		\underbrace{\int_{D_N} w(\omega)^{-s} |\hat{f_\eps} - \widehat{\PiN{f_\eps}}|^2\dom}_{\text{I}}
			+ \underbrace{\int_{D_N^c} w(\omega)^{-s} |\hat{f_\eps}|^2\dom}_{\text{II}}.
	\end{equation}

	For the second term, as before, we obtain
	\begin{align}
		\int_{D_N^c} w(\omega)^{-s}|\hat{f_\eps}|^2\dom &=
				\int_{D_N^c} \frac{\abs{\omega}^{2s}}{(1+\abs{\omega}^2)^{s}}  \abs{\hat\mollif(\eps\omega)}^2\abs{\hat{u}}^2\dom
			\leq \int_{D_N^c} |\omega|^{2s} \abs{\hat{u}}^2\dom .
	\end{align}
	For the first term, we follow the derivation of \cref{eq:cs_fourier_sum}.
	The difference  $f_\eps - \PiN f_\eps$ is transformed into a sum of $\hat f$ over
	a grid using the  Poisson  summation formula. 
	Again, the Poisson summation formula is applicable to $f_\eps$ due to the following considerations. For
	$\hat f_\eps$, the decay condition holds true
	using essentially the same argument as for $\hat u_\eps$. As $\abs{\hat u}$ decays
	and $\hat\mollif_\eps$ decays exponentially,
	$f_\eps = \abs{\omega}^{2s}\hat\mollif_\eps \hat u$ satisfies the decay condition 
	required in \cref{thm:poisson_summation}. Further, $f_\eps$ is smooth, thus we
	only have to verify that $\abs{f_\eps}(x)$ decays fast enough when $\abs{x}\rightarrow \infty$. To verify this, note that $f_\eps(x) = \flap u_\eps(x)$ and for $x\not\in\supp u_\eps$,
	\begin{align*}
		\abs{\flap u_\eps(x)} &= C(d,s) \left|\int_{\supp u_\eps} \frac{u_\eps(y)}{\abs{x-y}^{d+2s}}\dy \right|
					\leq C(d,s) \norm{u}_{L^1(\Rd)} \sup\limits_{y\in\supp u_\eps} \abs{x-y}^{-d-2s} \\
				&\leq C \norm{u}_{L^1(\Rd)} \abs{x}^{-d-2s}
	\end{align*}
	for $\abs{x}$ large enough, from which we obtain the required condition.
	
	The square
	of this sum, then, is transformed into a product of two sums using the Cauchy-Schwarz inequality after
	multiplying each summand with an appropriately written unity. Viz, with $w(\omega) \coloneqq (1+\abs{\omega}^2)$,
	\begin{align}
		\label{eq:Hms_err_f_2d}
		\qquad&\int_{D_N} w(\omega)^{-s} \Big( \sum\limits_{k\neq 0} \hat f_\eps(2N\pi k + \omega)\Big)^2\dom \nonumber\\
			& \leq (2N\pi)^{d}\int\limits_{[-1/2,1/2]^d} w(2N\pi\omega)^{-s}
				\sum\limits_{k\neq 0} \hat \mollif(\eps2N\pi (k + \omega))^2|2N\pi(k+\omega)|^{2s} \nonumber\\
				&\hspace{.38\textwidth}\cdot\;\sum\limits_{k\neq 0} \hat f(2N\pi (k + \omega))^2 |2N\pi(k+\omega)|^{-2s}\dom\nonumber \\
		&\leq \sup\limits_{\omega\in[-1/2,1/2]^d} \frac{(2N\pi)^{2s}}{w(2N\pi\omega)^s} 
			\sum\limits_{k\neq 0}\hat\mollif(2\pi\underbrace{\eps N}_{=1} (k + \omega))^2|(k+\omega)|^{2s} \cdot \nonumber\\
			&\hspace{.3\textwidth}(2N\pi)^{d}\int\limits_{[-1/2,1/2]^d}
					\sum\limits_{k\neq 0} \hat f(2N\pi (k + \omega))^2 |2N\pi(k+\omega)|^{-2s}\dom \nonumber\\
			&\leq C\int\limits_{D_N^c}|\hat f|^2 \abs{\omega}^{-2s}\dom = 
					C\int\limits_{D_N^c}\abs{\omega}^{2s}\abs{\hat u}^2 \dom
	\end{align}
	The integral-factor can be bounded exactly as before and has the desired rate
	of $N^{s-t}$, see the proof of \Cref{lemma:sinc_interp_smoothing} again. The $\sup$-factor $C$ reads 
	\begin{equation}
		\label{eq:constant_sup_Hms_err}
		C = \sup\limits_{\omega\in[-1/2,1/2]^d} \frac{(2N\pi)^{2s}}{w(2N\pi\omega)^s} 
			\sum\limits_{k\neq 0}\hat\mollif\left(2\pi(k + \omega)\right)^2|(k+\omega)|^{2s}.
	\end{equation}
	For a general mollifier, this unfortunately diverges at the rate $N^{2s}$. However, if we
	construct a mollifier in a way that the sum in \cref{eq:constant_sup_Hms_err}
	is sufficiently small in a neighborhood of  ${\omega = 0}$, the factor is bounded.
	In \Cref{sub:construction_of_mollifiers} we show how
	appropriate mollifiers, with support $[-\eps,\eps]^d\subset B_{\sqrt{d}\eps}(0)$
	can be constructed, thus completing the proof.
\end{proof}

We proceed with \textbf{Step 2}.
One can not in all cases compute $(\mollif_\eps * f)(x_k)$
for $x_k \in \Omega$, depending on the domain. Problems arise in particular if
$\supp \mollif_\eps(\cdot - x_k)\not\subset \Omega$. For an illustration, see
\Cref{fig:domain_mollif_exceed}.

We first consider the trivial case. \textit{If $\Omega$ is a one-dimensional interval,} the problem can be scaled to the domain $\Omega = (0,1)$.
Then, for $k = 1,\ldots,N-1$, we see that $(\frac{k-1}{N}, \frac{k+1}{N})\subset \Omega$
and the smoothing operation can easily be computed. The case of $\Omega = (0,1)^d, d > 1$, 
works in a similar fashion.

For a more general domain $\Omega$, we set $\rho = \sqrt{d}h$ and define an enlarged $\Omega_\rho = B_\rho(0) + \Omega$ (in the sense of a Minkowski sum) such that $B_\rho(k/N)\subset\Omega_\rho$ for all $k/N\in\Omega$. We then choose an appropriate right-hand side $f_\rho$ on $\Omega_\rho$. The exact choice of $f_\rho$ is discussed in \Cref{sub:domain_approximation}. 

Once this is done, we define the solution $u_\rho\in\widetilde H^s(\Omega_\rho)$ as the solution to
\begin{equation}
	\label{eq:def_f_rho_bar_i}
	\flap u_\rho = f_\rho\text{ in }\Omega_\rho,\quad u_\rho \equiv 0 \text{ in }\Omega_\rho^c
\end{equation}
and set
\begin{equation}
	\label{eq:def_f_rho_bar_ii}
	\bar f_\rho(x) = \flap u_\rho(x)\text{ on } \R^d.
\end{equation}
Given this extension, we obtain our right-hand-side as the
$\sinc$-interpolation of the smooth approximation of the enlarged right-hand side, namely
\begin{equation}
	f_h \coloneqq \PiN \left(\mollif_\eps * \bar f_\rho\right) = \PiN \bar f_{\rho,\eps}.
\end{equation}
With this in mind, the numerator of the consistency error can be expressed as
\begin{align}
	a(u - u_h, w_h) &= a(u - u_\rho + u_\rho - u_{\rho,\eps} + u_{\rho, \eps} - u_h, w_h) \nonumber\\
		&\leq \bigg(\underbrace{\Hs{u - u_\rho}}_{\eqqcolon\text{I}} + 
				\underbrace{\Hs{u_\rho - u_{\rho,\eps}}}_{\eqqcolon\text{II}}\bigg)\cdot \Hs{w_h} 
				\nonumber\\
		&\qquad\qquad\qquad\qquad\qquad	+ \underbrace{\Hms{\bar f_{\rho, \eps} - \PiN \bar f_{\rho, \eps}}}_{\eqqcolon\text{III}}
			\cdot \norm{w_h}_{H^s(\Rd)}\label{eq:cons_term_three_split_i}.
\end{align}
Altogether, we have three terms that we need to bound with appropriate rates.
Both of the terms (II) and (III) can be bounded in terms of the $H^t$-seminorm of $u_\rho$
with the rate $h^{t-s}$ using \Cref{lemma:sinc_interp_smoothing,lemma:sinc_interp_smoothing_f}
and setting $\eps = h$. What is still open is the term (I), whose bound we 
discuss in the following \cref{sub:domain_approximation}.

\subsection{Domain Enlargement}
\label{sub:domain_approximation}
We have to choose $f_\rho$ such that if $u$ and $u_\rho$ solve
\[
	\flap u = f\text{ in }\Omega,\quad u\equiv 0 \text{ in }\Omega^c
\]
and
\[
	\flap u_\rho = f_\rho\text{ in }\Omega_\rho,\quad u\equiv 0 \text{ in }\Omega^c, 
\]
then
\begin{equation}
    \Hs{u-u_\rho} \leq C \rho^{1/2}.
\end{equation}
As we assumed that $f \in C^{\beta}(\overline\Omega)$ and $\Omega$ Lipschitz, we know that
there exists $f_\rho \in C^{\beta}(\Omega_\rho)$ such that $f = f_\rho$ in $\Omega$, see \cite[Lemma 6.37]{Gilbarg2001}.
For this $f_\rho$, we obtain $\bar f_\rho$ on $\Rd$ as described in \cref{eq:def_f_rho_bar_i,eq:def_f_rho_bar_ii}. For the estimate of $\Hs{u-u_\rho}$, we want 
to prove the following Theorem.

\begin{theorem}
	\label{thm:u_uR_Hs}
	Let $\Omega_\rho$ be the extended domain  and $f_\rho\in L^\infty(\Omega_\rho)$ 
	a Hölder continuous extension of the given right-hand side $f$ to the domain $\Omega_\rho$ as discussed before . Let $u_\rho$
	solve the exterior value problem
	\begin{equation}
		\flap u_\rho = f_\rho\text { in }\Omega_\rho,\quad u_\rho = 0\text{ in }\Omega^c
	\end{equation}
	and $u$ the solution to \cref{eq:bvp}. Then,
	\begin{equation}
		\Hs{u-u_\rho} \leq C h^{1/2},
	\end{equation}
	where $C$ is a constant that depends on $f_\rho,\Omega\text{ and }d$.
\end{theorem}

In order to prove this, we first prove the following lemma.
\begin{figure}
	\begin{center}
		\includegraphics[height=5cm]{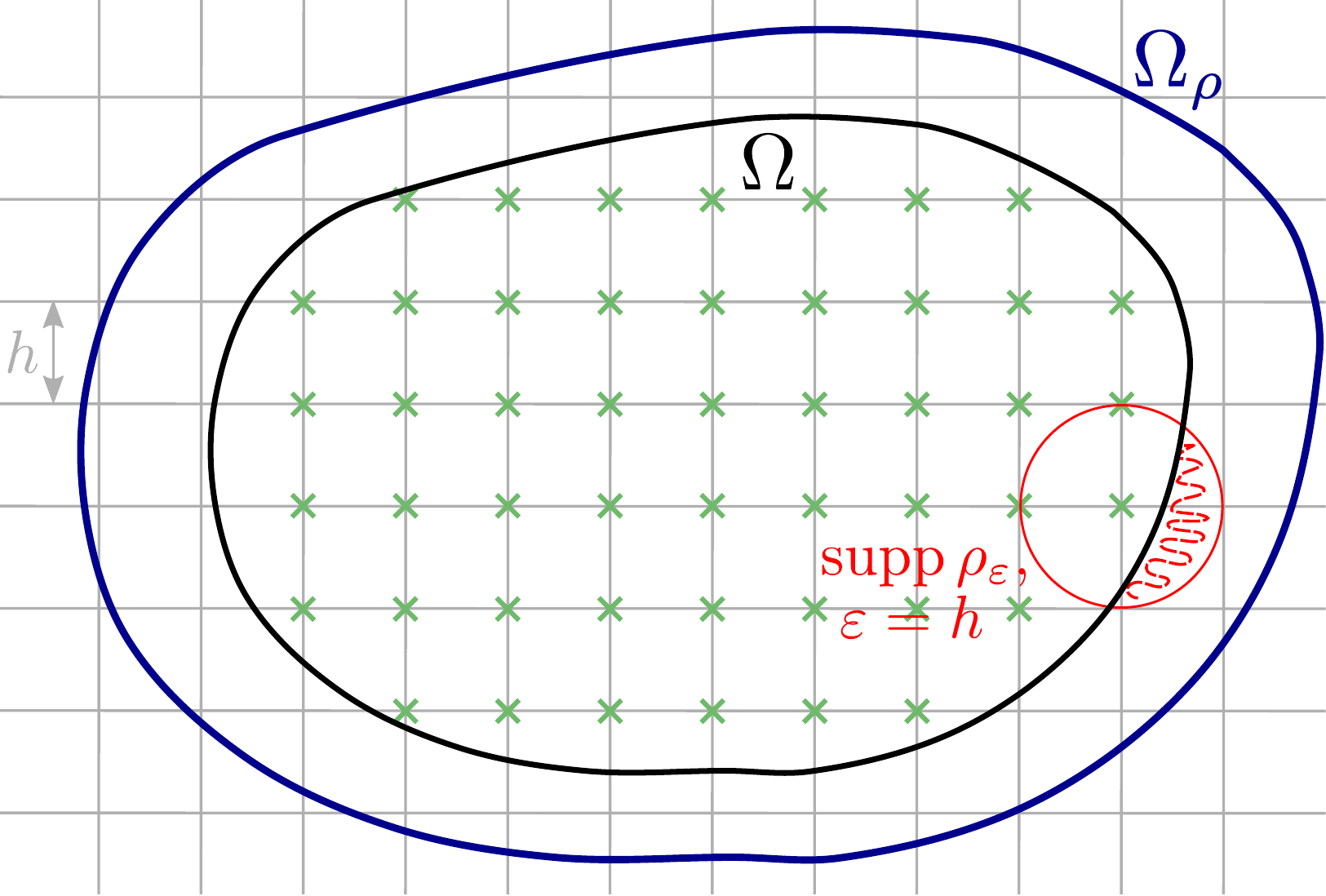}
	\end{center}
\caption{In more than one dimension, it may happen that the support of the smoothing
			kernel exceeds the boundary of $\Omega$.}
\label{fig:domain_mollif_exceed}
\end{figure}

\begin{lemma}
	\label{lemma:outer_flap_behav}
	Let $u$ solve \cref{eq:bvp} on a bounded Lipschitz domain that fulfills the
	exterior ball condition and $f\in L^{\infty}(\Omega)$. For $x\in\Omega^c$, let
	\begin{equation}
		\label{eq:def_f_outside_Om}
		\bar f(x) = \Ns u(x) = -C(d,s)\int_\Omega \frac{u(y)}{\abs{x-y}^{d+2s}}\dy. 
	\end{equation}
	Then,
	\begin{equation}
		\abs{\bar f(x)} \leq \cbound(\Omega, s)\cdot \delta(x)^{-s},\quad \delta(x) = \dist(x, \delOm)
	\end{equation}
	where $\cbound(\Omega, s)$ is the constant from \Cref{thm:bound_u_boundary}.
\end{lemma}
\begin{proof}
	We know that for $y\in\Omega$ (see \cref{eq:bound_u_boundary}) 
	\begin{equation} 
		u(y) \leq C(\Omega, s) \norm{f}_{L^\infty(\Omega)} \dist(y, \delOm)^s.
	\end{equation}
	Further, as $x\not\in\Omega$, 
	\begin{equation}
		\dist(y,\delOm) = \inf\limits_{z\in\delOm} \abs{y-z} \leq \abs{y-x},
	\end{equation}
	thus
	\begin{equation}
		\label{eq:bound_f_integrand_outer}
		\frac{u(y)}{\abs{x-y}^{d+2s}} 
			\leq C \frac{\dist(y, \delOm)^{s}}{\abs{x-y}^{d+2s}}
			\leq C \frac{\abs{x-y}^{s}}{\abs{x-y}^{d+2s}}
			= C \frac{1}{\abs{x-y}^{d+s}}. 
	\end{equation}
	 We plug this into \cref{eq:def_f_outside_Om} to obtain
	\begin{equation}
		\abs{\bar{f}(x)} \leq C \int_\Omega \frac{1}{\abs{x-y}^{d+s}}\dy. 
	\end{equation}
	To compute this integral, assume w.l.o.g $x = 0\not\in\overline{\Omega}$ and choose $R > 0$ large enough
	such that $\Omega\subset B_R(0)$, then 
	
	\begin{align}
		\int_\Omega \frac{1}{\abs{x-y}^{d+s}}\dy = \int_{B_R}{\abs{y}^{-(d+s)}} \chi_\Omega(y)\dy 
			&= \int_0^R\int_{\mathbb S^{d-1}} {\abs{r\xi}^{-(d+s)}} \chi_\Omega(r\xi)\doxi r^{d-1}\dr \nonumber\\
			&= \int_0^R r^{-1-s}\int_{\mathbb S^{d-1}} 
			\underbrace{\chi_\Omega(r\xi)}_{=0\text{ if }r<\delta(x)}\doxi \dr\nonumber\\
			&\leq \abs{\mathbb S^{d-1}} \frac{1}{-s} \left( R^{-s} - \delta(x)^{-s} \right) \leq  \frac{\abs{\mathbb S^{d-1}}}{s}\cdot \delta(x)^{-s}.
	\end{align}
	
\end{proof}

We are ready to prove \Cref{thm:u_uR_Hs} now.
\begin{proof}[Proof (of \Cref{thm:u_uR_Hs})]
	We split the error into two integrals:
	\begin{align}
		\Hs{u-u_\rho}^2 &\leq\int_\Rd \underbrace{|\flap(u-u_\rho)|}_{=0 \text{ in }\Omega} 
											\cdot \underbrace{|u-u_\rho|}_{=0 \text{ in }\Omega_\rho^c}\dx\nonumber \\
			&= \int_{\Omega_\rho\setminus\Omega}
			\big|\flap u-\underbrace{\flap u_\rho}_{=f_\rho}\big|
							\cdot |\underbrace{u}_{=0} - u_\rho|\dx \nonumber\\
			&\leq \int_{\Omega_\rho\setminus\Omega} |f_\rho|\cdot |u_\rho| \dx
					+ \int_{\Omega_\rho\setminus\Omega} |\flap u|\cdot |u_\rho|\dx\label{eq:u_uR_Hs_split}.
	\end{align}
	We bound the integrals individually and obtain for the first integral
	\begin{align}
		\int_{\Omega_\rho\setminus\Omega} |f_\rho|\cdot \abs{u_\rho} 
	    &\leq \norm{f_\rho}_{L^{\infty}(\Omega_\rho)}\cdot
					\int_{\Omega_\rho\setminus\Omega} \abs{u_\rho}\dx\nonumber \\
	    &\leq\norm{f_\rho}_{L^{\infty}(\Omega_\rho)}\cdot
			\cbound(\Omega_\rho, f_\rho)\sup\limits_{x\in\Omega_\rho}
				\underbrace{\dist(x,\partial\Omega_\rho)^s}_{\leq h^s}
					\int_{\Omega_\rho\setminus\Omega} 1\dx \nonumber\\
		&\leq \norm{f_\rho}_{L^{\infty}(\Omega_\rho)}\cdot\cbound(\Omega_\rho, f_\rho)\cdot h^s \cdot C(\Omega) h \label{eq:u_uR_Hs_split_i}
	\end{align}
	because, with 
	\begin{equation}
		\Omega_\rho\setminus\Omega = \bigcup\limits_{t\in(0,h)} \Gamma_t,\quad
			\Gamma_t\coloneqq\left\{
				x\in\Omega_\rho\setminus\Omega\, | \, \dist(x,\delOm) = t
					\right\},
	\end{equation}
	we have
	\begin{align}
	\int_{\Omega_\rho\setminus\Omega} 1\dx = \int_0^h \int_{\Gamma_t} 1 \dx\dt
	    \leq \int_0^h\abs{\Gamma_h} h \dt \leq C(\Omega)\cdot h,
	\end{align}
	where $C(\Omega)$ is a constant that depends only on $\Omega$.
	We reparameterize the second integral using the same parameterization of $\Omega_\rho\setminus\Omega$
	and \Cref{lemma:outer_flap_behav} to obtain 
	\begin{align}
		\int_{\Omega_\rho\setminus\bar\Omega} \abs{\flap u}\cdot \abs{u_\rho}\dx &=
			\int\limits_{0}^h\int_{\Gamma_t} \abs{\flap u}\cdot \abs{u_\rho} \dx\dt 
			\leq \cbound(\Omega_\rho, f_\rho)\cdot h^{s} \int\limits_{0}^h \int_{\Gamma_t} t^{-s} \dx\dt \nonumber \\
			&= \cbound(\Omega_\rho, f_\rho)\cdot h^s \int\limits_{0}^h\abs{\Gamma_t} t^{-s}\dt \nonumber\\
			& \leq \cbound(\Omega_\rho, f_\rho)\cdot h^s C(\Omega) h^{1-s} 
			= \cbound(\Omega_\rho, f_\rho) C(\Omega) h  \label{eq:u_uR_Hs_split_ii}
	\end{align}
We plug \cref{eq:u_uR_Hs_split_i,eq:u_uR_Hs_split_ii} into \cref{eq:u_uR_Hs_split} to obtain
\begin{align}
	\Hs{u - u_\rho}^2 &\leq C(\Omega)\cdot \cbound(\Omega, f_\rho) \cdot h^{s+1} 
			+ C(\Omega)\cdot \cbound(\Omega, f_\rho)\cdot h \nonumber\\
				&\leq 2\, C(\Omega)\cdot \cbound(\Omega, f_\rho)\cdot h
\end{align}

	where we used \cref{ass:dom} which ensures that $\cbound(\Omega_\rho, f\rho)\leq C\cdot \cbound(\Omega, f\rho)$.
This completes the proof.

\end{proof}

\section{Main theorems}
\label{sec:main_theorems}
Using the results proven so far,  we are finally ready to prove our main theorems.
\begin{theorem}
	\label{thm:main_i}
	Let $\Omega$ be a bounded Lipschitz domain that fulfills Assumption \ref{ass:dom}.  Let $f$ be given as in \Cref{thm:sobolev_reg_u_lipschitz} and
	$u$ solve the exterior value problem \cref{eq:bvp} on $\Omega$.
	Let $u_h\in \V_h(\Omega)$ solve the discrete problem, see
	\cref{eq:formulation_collocation}, where the right-hand side is chosen as
	\[
		f_k = \bar f_{\rho,\eps}(x_k)
	\]
	for $k/N\in\Omega$. Let $\rho = \sqrt{d}h$ and $\Omega_\rho$, $f_\rho$, $u_\rho$
	and $\bar f_\rho$ as described in \Cref{sub:domain_approximation} 
	and $\bar f_{\rho,\eps}$ is obtained as a mollification of $\bar f_\rho$,
	see \cref{eq:def_f_bar}, with the mollifier constructed in \cref{sub:construction_of_mollifiers}.
	Then,  with $\theta = 1/\abs{\log(h)}$ as in \cite{Borthagaray2017},
		\begin{align*}
			\norm{u-u_h}_{H^{s}(\Rd)} &\leq C\cdot\Big(
				\csob(\Omega, f) + \cbound(\Omega, f) + 
				\csob(\Omega_\rho, f_\rho) + \cbound(\Omega_\rho, f_\rho) \Big)
				\cdot \frac{h^{1/2-\theta}}{\theta} \\
				&=C\cdot\Big(
				\csob(\Omega, f) + \cbound(\Omega, f) + 
				\csob(\Omega_\rho, f_\rho) + \cbound(\Omega_\rho, f_\rho) \Big) \abs{\log(h)} h^{1/2}
		\end{align*}
	where $C$ depends on $\Omega, s$ and $d$.
\end{theorem}
\begin{proof}
	This is a direct consequence of the results proven so far. Using the abstract error estimate,
	see \Cref{thm:abstract_error}, we obtain that 
	\begin{equation}
		\norm{u-u_h}_{H^{s}(\Rd)} \leq
		\left( 1 + \frac{C}{\gamma}  \right) \inf\limits_{v_h\in \V_h(\Omega)} \norm{u - v_h}_{H^s(\Rd)}
			+ \frac{1}{\gamma}  \sup\limits_{w_h\in \V_h(\Omega)\setminus\{0\}} \frac{\abs{a(u-u_h, w_h)}}{\norm{w_h}_{H^s(\Rd)}}
	\end{equation}
	The bound for the first term is a direct consequence of \Cref{prop:bound_approximation_error}.
	The numerator of the second term can be bounded after adding appropriate zeros
	and using the triangle inequality by
	\begin{equation}
		\abs{a(u-u_h, w_h)} \leq  \left(\Hs{u-u_\rho} + \Hms{\bar f_\rho - \bar f_{\rho,\eps}}
			+ \Hms{\bar f_{\rho,\eps} - \PiN f_{\rho,\eps}}\right)\cdot \norm{w_h}_{H^s(\Rd)}
	\end{equation}
	Now, the desired bound follows by applying \Cref{thm:u_uR_Hs} to the first term
	of the sum and \Cref{lemma:sinc_interp_smoothing_f} to the second and third term,
	and further noticing that we set $\rho = \sqrt{d}h$.
\end{proof}

We acknowledge that although, in principle, \cref{thm:main_i} gives an appropriate
error bound, actually obtaining a right-hand side with the extension-and-mollification
approach may be challenging in practice. However, if we assume that $f$ has some
more regularity in $\Omega$, we can enhance the previous result and show that
we can indeed obtain the right-hand side of the discrete system by simple direct
sampling, i.e., setting $f_k = f(k/N)$ for $k/N\in\Omega$.
\begin{theorem}
	\label{thm:main_ii}
	Let the setting be as in \Cref{thm:main_i} and let, additionally, $f\in C^{1/2}(\overline{\Omega})$.
	Let the entries of the right-hand side in \cref{eq:formulation_collocation} be
	chosen as
	\[
		f_k = f(k/N)
	\]
	for $k/N\in\Omega$. Then,
	\begin{align*}
		\norm{u-u_h}_{H^{s}(\Rd)} &\leq C\cdot\Big(
			\csob(\Omega, f) + \cbound(\Omega, f) + 
			\csob(\Omega_\rho, f_\rho) + \cbound(\Omega_\rho, f_\rho) \Big)
			\cdot \frac{h^{1/2-\theta}}{\theta}\\
	&=C\cdot\Big(
			\csob(\Omega, f) + \cbound(\Omega, f) + 
			\csob(\Omega_\rho, f_\rho) + \cbound(\Omega_\rho, f_\rho) \Big)  \abs{\log(h)} h^{1/2}
	\end{align*}
	where $C$ depends on $\Omega, s$ and $d$.
\end{theorem}
\begin{proof}
	The arguments are in essence the same as in the proof of \Cref{thm:main_i}. The additional
	term we have to bound in the consistency term is
	\begin{equation}
	\label{eq:bound_Pif_Pifeps}
		(\PiN f_\rho - \PiN f_{\rho,\eps}, w_h)_{L^2(\Rd)}
			\leq\left( N^{-d}\sum\limits_{k/N\in\Omega} (f_\rho(x_k) - f_{\rho,\eps}(x_k))^2\right)^{1/2}
				\cdot \Lt{w_h}
	\end{equation}
	First, note that we extend $f$ to $f_\rho\in C^{1/2}(\Omega_\rho)$ as before. Then,
	for $x_k\in\Omega$, 
	\begin{align*}
		\abs{f_{\rho,\eps}(k/N) - f_\rho(k/N)} &= \left|\int_{B_\eps(x)} \eta_\eps(x-y)(f_\rho(x) - f_\rho(y))\dy\right| \\
		&\leq \int_{B_\eps(x)} \eta_\eps(x-y)\dy \sup\limits_{y\in B_\eps(x)} |f_\rho(x) - f_\rho(y)| \\
		&\leq C \abs{x-y}^{1/2} \leq C \eps^{1/2}.
	\end{align*}
	Plugging this into \cref{eq:bound_Pif_Pifeps} and setting $\eps = h$ gives the
	desired result.
\end{proof}
If the domain has not only a Lipschitz boundary but a smooth boundary, we can obtain
error boundaries also without the assumption that $f$ is Hölder continuous.
\begin{theorem}
	Let $\Omega$ be a smooth domain that fulfills Assumption \ref{ass:dom} and
	let $f\in L^\infty(\Omega)$. Let $u$ solve the exterior value problem \cref{eq:bvp} on $\Omega$. 
	Let $u_h\in \V_h(\Omega)$ solve the discrete problem, see
	\cref{eq:formulation_collocation}, where the right-hand side is chosen as
	\[
		f_k = \bar f_{\rho,\eps}(x_k).
	\]
	Then,
	\[
		\norm{u-u_h}_{H^{s}(\Rd)} \leq C\cdot\Big(
			\csob(\Omega, f) + \cbound(\Omega, f) + 
			\csob(\Omega_\rho, f_\rho) + \cbound(\Omega_\rho, f_\rho) \Big)
			\cdot \frac{1}{\theta} h^{\alpha}.
	\]
	In case $s < \frac{1}{2}$, we have $\alpha = s$ and $C$ is a constant
	that depends on $\Omega, d$ and $s$.

	In case $s \geq \frac{1}{2}$, we have $\alpha = \frac{1}{2} - \theta$ for all 
	$\theta > 0$ and $C$ is a constant that depends on $\Omega, d, s$ and $\alpha$.
\end{theorem}
\begin{proof}
	The proof is essentially the same as for \Cref{thm:main_i,thm:main_ii} using 
	\Cref{prop:bound_approximation_error_2} and that we can extend $f\in L^\infty(\Omega)$
	to $f_\rho\in L^\infty(\Omega_\rho)$.
\end{proof}
\begin{remark}
	The main ingredients to prove the theorems above are Sobolev regularity and decay near the boundary of solutions $u$. Thus it is easily possible to derive similar results for specific cases where regularity and decay estimates are also available,  with sub-error decay rates depending on regularity and boundary decay rate. A possible application are convergence results on (not necessarily smooth) Lipschitz domains that fulfill the exterior ball condition for $f\in L^\infty(\Omega)\subset L^2(\Omega)$ using the results shown in \cite{Borthagaray2020,Borthagaray2021}
\end{remark}

\section*{Acknowledgement}
We thank Prof. Ricardo H. Nochetto of University of Maryland College Park for fruitful discussions while we were preparing this manuscript. LS and PD gratefully acknowledge the hospitality of George Mason University.

\appendix

\section{Construction of appropriate mollifiers}
\label{sub:construction_of_mollifiers}
The above estimates for the approximation of the right-hand side only give us the desired rate if we can bound the constant $C$ in
\cref{eq:constant_sup_Hms_err} appropriately. This is achieved
if the sum 
\begin{equation}
	\label{eq:sum_mollifier_suppl}
	\sum_{k\neq 0}\hat\mollif\left(2\pi(k + \omega)\right)^2|(k+\omega)|^{2s}
\end{equation}
(i) is finite and
(ii) sufficiently small around ${\omega = 0}$.
We achieve this as follows. We want to have a mollifier $\mollif(x)$ such that
\begin{equation}
	\label{eq:def_sinc_prod_mollif}
	\hat\mollif(\omega) = \hat \psi(\omega)\cdot S(\omega)
\end{equation}
where
\begin{equation}
	\label{eq:def_sinc_prod_mollif_2}
	S(\omega) 
        = \prod\limits_{i=1}^d \sinc(\omega_i/(2\pi))
\end{equation}
and $\hat\psi(\omega)\rightarrow 0$ as $\abs{\omega}\rightarrow \infty$ fast enough 
to ensure that the sum in \cref{eq:sum_mollifier_suppl} remains bounded.
Precisely, the mollifier $\mollif$ is obtained as the convolution of
$\psi$ and $\check S \coloneqq\IFT S$, which gives the structure of the
mollifier, see \cref{eq:def_sinc_prod_mollif}, through the convolution theorem
for the Fourier transformation. Set
\begin{equation}
	\check S = \chi_{(-1/2,1/2)^d}(x) = \begin{cases}
		1&\text{if } x\in (-1/2,1/2)^d\\
		0&\text{otherwise}
	\end{cases},
\end{equation}
the indicator function on the cube from $-1/2$ to $1/2$ in each spatial direction. Note that for this definition,
$S = \FT{\check S}$ is exactly fulfilled.
It is clear that $\supp{\check S}\subset B_{\sqrt d/2}$. Thus clearly, if we choose
$\psi$ to be a smooth mollifier and scale its support to a circle of radius $\sqrt{d}/4$,
we have $\supp{\check S * \psi}\subset B_{\sqrt{d}}$.
Further, we have to normalize $\psi$ such that $\int_\Rd \left(\psi * \check S\right) \dx = 1$.
As we assume $\psi\subset C_c^\infty(\Rd)$, we can assume $\hat\psi(\omega) \leq C\exp(-\omega)$.
The so constructed mollifier in $d = 2$ is shown in \cref{fig:constr_mollifier_2d}.

The second factor in \cref{eq:def_sinc_prod_mollif} ensures that the sum in \cref{eq:constant_sup_Hms_err} is $0$ if $\omega = 0$,
the first ensures that it remains bounded. To show this,
note that for the mollifier as defined in equations \cref{eq:def_sinc_prod_mollif,eq:def_sinc_prod_mollif_2},
it holds that
\begin{align}
	\sum\limits_{k\neq 0} \hat \mollif(2\pi (k + \omega))^2|k+\omega|^{2s}
		&= \sum\limits_{k\neq 0} \left(S(2\pi(k+\omega)) \hat\psi(2\pi (k + \omega))  \right)^2|k+\omega|^{2s} \nonumber\\
		&= \sum\limits_{k\neq 0}
			\left(\prod\limits_{i=1}^d \frac{\sin(\pi(k_i+\omega_i))}{\pi(k_i+\omega_i)}\right)^2
			\hat\psi^2(2\pi (k + \omega))  |k+\omega|^{2s} \nonumber\\
		&=  \prod\limits_{i=1}^d  \left(\frac{\sin(\pi\omega_i)}{\pi}\right)^2 \sum\limits_{k\neq 0}
			\prod\limits_{i=1}^d \frac{1}{(k_i+\omega_i)^2}
			\hat\psi^2(2\pi (k + \omega))  |k+\omega|^{2s} \nonumber\\
		&\leq C	\prod\limits_{i=1}^d \left(\frac{\sin(\pi\omega_i)}{\pi}\right)^2  \sum\limits_{k\neq 0}
			 \frac{ \exp(-4\pi \abs{k+\omega}) \cdot |k+\omega|^{2s}}{\prod_{i=1}^d(k_i+\omega_i)^2} \nonumber\\
		&\leq C \prod\limits_{i=1}^d \omega_i^2  \underbrace{\sum\limits_{k\neq 0}
			 \frac{ \exp(-4\pi \abs{k+\omega}) \cdot |k+\omega|^{2s}}{\prod_{i=1}^d(k_i+\omega_i)^2}}_{\leq\text{const.}} \nonumber\\
		&\leq C \prod\limits_{i=1}^d \omega_i^2 \leq C \abs{\omega}^{2d}
\end{align}
where $C$ is a varying constant that depends on the decay of $\hat\psi$, $s$ and $d$.
For the last inequality, we used the generalized mean inequalities. 
We plug this into \cref{eq:constant_sup_Hms_err} and obtain
\begin{equation}
	\frac{(2N\pi)^{2s}}{w(2N\pi\omega)^s}
			\sum\limits_{k\neq 0}\hat\mollif\left(2\pi(k + \omega)\right)^2|(k+\omega)|^{2s}
			\leq C \frac{1}{\left(1/(2N\pi)^2+\abs{\omega}^2\right)^{s}} \abs{\omega}^{2d}
			\leq \abs{\omega}^{2d-2s} < \infty
\end{equation}
if $s < d$.
\begin{figure}
	\centering
	\includegraphics{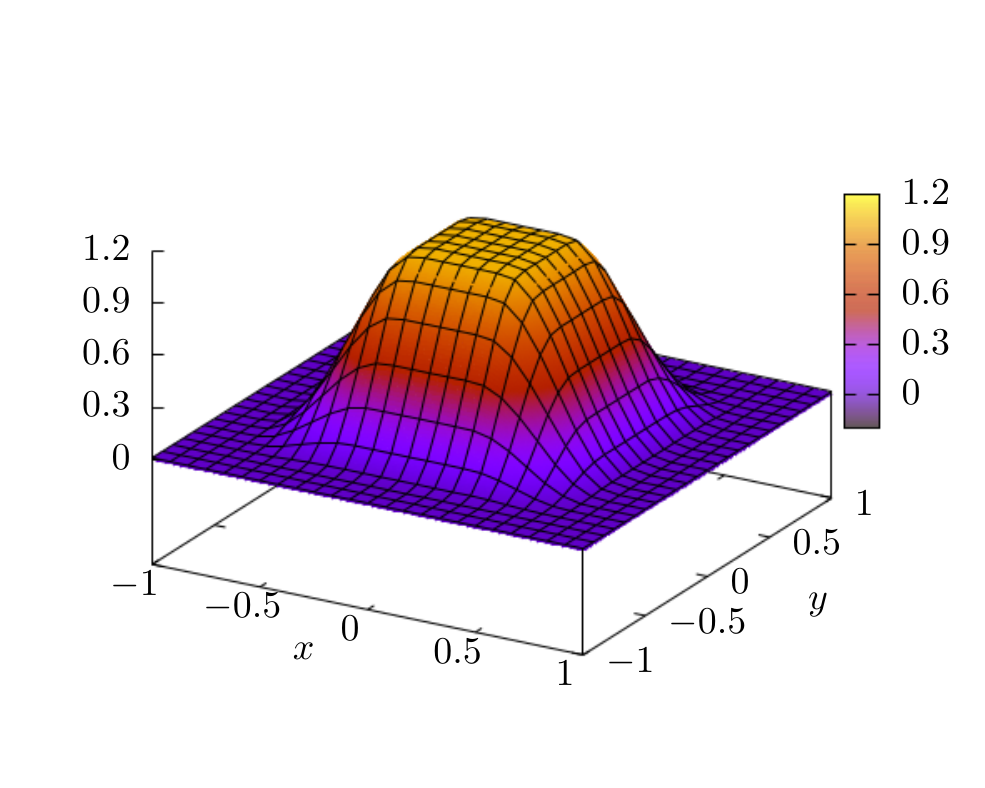}
	\caption{The mollifier constructed as $\eta(x) = (S*\psi)(x)$ and normalized s.t. $\int_\Rd \eta\dx = 0$ in $d = 2$.}
	\label{fig:constr_mollifier_2d}
\end{figure}

\bibliographystyle{plain}	
\bibliography{bibliography}

\end{document}